\documentclass[reqno,10pt]{amsart}
\usepackage{amsmath}
\usepackage{amssymb}
  \usepackage{paralist}
  \usepackage{graphics} 
  \usepackage{epsfig} 
\usepackage{graphicx}  \usepackage{epstopdf}
 \usepackage[colorlinks=true]{hyperref}
\hypersetup{urlcolor=blue, citecolor=red}

\newtheorem{theorem}{Theorem}[section]

\newtheorem{lemma}[theorem]{Lemma}
\newtheorem{proposition}{Proposition}

\theoremstyle{definition}
\newtheorem{definition}[theorem]{Definition}
\newtheorem{remark}{Remark}

\newcommand{\tc }{\mathtt{c}}
\newcommand{\tb}{\mathtt{b}}

\title[Transf. of En. through fast diff. channels in res. PDEs on the circle] 
      {Transfers of energy through fast diffusion channels in some resonant PDEs on the circle}

\author[Filippo Giuliani]{}


 \email{filippo.giuliani@upc.edu}

\thanks{The author is supported by the European Research Council (ERC) 
under the European Union's Horizon 2020
research and innovation programme under grant agreement 
No 757802.}

\begin{document}
\maketitle

\centerline{\scshape Filippo Giuliani$^*$}
\medskip
{\footnotesize
 \centerline{Departament de Matematiques, Universitat Politecnica de Catalunya,}
 \centerline{ETSEIB, Avinguda Diagonal 647}
   \centerline{08028 Barcelona, Spain}
} 

\begin{abstract} 
In this paper we consider two classes of resonant Hamiltonian PDEs on the circle with non-convex (respect to actions) first order resonant Hamiltonian. 
We show that, for appropriate choices of the nonlinearities we can find time-independent linear potentials that enable the construction of solutions that undergo a prescribed growth in the Sobolev norms.  
The solutions that we provide follow closely the orbits of a nonlinear resonant model, which is a good approximation of the full equation. The non-convexity of the resonant Hamiltonian allows the existence of \emph{fast diffusion channels} along which the orbits of the resonant model experience a large drift in the actions in the optimal time. This phenomenon induces a transfer of energy among the Fourier modes of the solutions, which in turn is responsible for the growth of higher order Sobolev norms.
\end{abstract}

\tableofcontents

\section{Introduction}

We consider the following resonant Hamiltonian PDEs under periodic boundary conditions:
\begin{itemize}
\item Nonlinear wave equations with even-power nonlinearity
\begin{equation}\label{eqw}
u_{tt}-\Delta u+V_p * u+ u^p=0, \quad u=u(t, x), \quad t\in\mathbb{R}, \quad x\in\mathbb{T}:=\mathbb{R}/2\pi \mathbb{Z}
\end{equation}
with $p$ even and large enough.
\item Nonlinear Schr\"odinger equations with cubic $x$-dependent nonlinearity
\begin{equation}\label{eqn}
\mathrm{i} u_t-\Delta u+V_N*u+\cos(N x) |u|^2 u=0, \quad  u=u(t, x), \quad t\in\mathbb{R}, \quad x\in\mathbb{T}:=\mathbb{R}/2\pi \mathbb{Z}
\end{equation}
with $N\geq 1$.
\end{itemize} 
We show that, prescribed an arbitrarily large (but finite) growth, for an opportune choice of the nonlinear terms and \emph{time-independent} potentials $V_p(x), V_N(x)$ (supported on few harmonics) we are able to construct arbitrarily small initial data solutions of the equations \ref{eqw}, \ref{eqn} whose Sobolev norms undergo the prescribed growth (we refer to Section \ref{secmainres} for the precise statements of the results).
 We point out that this phenomenon is purely nonlinear, since the potentials $V_p(x)$, $V_N(x)$ are chosen such that the origin $u=0$ is a resonant \emph{elliptic} fixed point - all the linear solutions are time-periodic or quasi-periodic functions that do not exchange energy among their modes.\\ The construction of solutions whose norms exhibit growth relies on the resonant nature of the equilibrium and the degeneracy of the first order resonant Hamiltonian.\\
We remark that in the linear setting one can obtain stronger results, such as the existence of unbounded orbits, by adding smooth time-dependent potentials to resonant equations, see for instance \cite{Maspero18g}.

\smallskip

In a neighborhood of a resonant elliptic equilibrium the analysis of the nonlinear dynamics can be performed through Birkhoff normal form methods. We can find a change of coordinates that removes from the Hamiltonian the terms which are not relevant for the dynamics of the equations in a certain range of times. In these coordinates the Hamiltonian is said to be in normal form.\\ 
We construct solutions which are close to orbits of a finite-dimensional \emph{resonant model}, which is obtained by restricting the Hamiltonian in normal form on an opportune submainfold of the phase space. The non-convexity of such Hamiltonian allows the existence of certain affine subspaces of the action space, called \emph{diffusion channels}. The orbits that travel on these channels are locked in a resonance and exhibit a relevant drift in (some of) the actions in the optimal time.\\
 A fast instability phenomenon of this kind is illustrated by the following example: consider the two-degrees of freedom Hamiltonian system 
\begin{equation}\label{example}
H(\theta_1, \theta_2, I_1, I_2)=I_1+\varepsilon \cos(\theta_2), \quad \theta_i\in\mathbb{T}, \,\,I_i\in(0, r), \,\,i=1, 2,
\end{equation}
for some $r>0$.
When $\varepsilon=0$ the phase space is foliated by $2$-dimensional invariant resonant tori filled by periodic orbits. Hence there is stability in the actions for all time.
For $\varepsilon>0$ the equations of motion are
\[
\dot{\theta}_1=1, \quad \dot{\theta_2}=0, \quad \dot{I}_1=0, \quad \dot{I}_2=\varepsilon \sin (\theta_2).
\]
We note that any section $\{\theta_2=\alpha\}$, $\alpha\in [0, 2\pi)$ is invariant. If we choose as initial conditions $\theta_2(0)=\alpha$ with $\alpha\notin \{0, \pi\}$ it is easy to see that $I_2(t)$ experiences a drift of order $\mathcal{O}(1)$ in a time $T=\mathcal{O}( \varepsilon^{-1})$. Thus all the periodic orbits on a given unperturbed torus, except two, are destroyed and give rise to diffusive solutions.\\
The first examples of finite-dimensional nearly-integrable Hamiltonian systems exhibiting fast instability have been provided by Nekhoroshev \cite{Nekhoroshev77}. The analysis of such systems has been carried out by Biasco-Chierchia-Treschev \cite{BCT2006}, Bounemoura-Kaloshin \cite{KalBoune} and Bounemoura \cite{Boune}. In these works the authors consider two degrees of freedom systems with an unperturbed Hamiltonian that violate the Nekhoroshev's condition for stability (see  \cite{Nekhoroshev77}) and study generic properties of the perturbations that provide fast diffusion. In the present paper we are interested in studying how these fast instability phenomena may be exploited, in the context of Hamiltonian PDEs under periodic boundary conditions, to obtain different ways to transfer energy among Fourier modes of nonlinear wave solutions.\\
The dynamics of the linearized problem at a resonant elliptic equilibrium is in some way similar to the unperturbed dynamics of the Hamiltonian in \ref{example}: there is plenty of resonant invariant tori supporting motions that fill densely lower dimensional submanifolds. As observed by Poincar\'e, these invariant objects are usually not robust, even under small perturbations. Then we may expect that, under some degeneracy assumptions on the Hamiltonian, some of them may be partially destroyed under the effect of the nonlinear terms and accomodate unstable behaviors.
 As a counterpart we mention that, under assumptions of integrability and non-degeneracy of the normal form Hamiltonian,
the existence of invariant tori very close to resonances, see for instance \cite{KdVAut}, \cite{FGP}, and results of long-time stability \cite{Bambusi99b}, \cite{BambusiNek} have been provided for completely resonant PDEs on tori.

\smallskip

A drift in the actions in an opportune symplectic reduction of the resonant model corresponds to a transfer of energy between resonant modes. 
An arbitrarily large growth of Sobolev norms can be obtained if this transition occurs across increasingly high Fourier modes.\\ 
In this paper we just consider solutions that display an exchange of energy among few modes.
Similar analysis for the dynamics of single resonant clusters have been recently carried out for the search of time-recurrent exchanges of energy, usually called \emph{beatings}, see for instance \cite{GT}, \cite{HausP17}. 
In the present paper we consider "one-way" transfers of energy between two sets of modes, say from \emph{low} to \emph{high} modes (see the end of Section \ref{secmainres} for a comparison with periodic beatings).
Roughly speaking, the nonlinearities play the role of \emph{external parameters} that modulate the size of the high modes. This allows to obtain the desired growth in the Sobolev norms by choosing appropriately the nonlinear terms. Since the construction of the solutions that we provide is relatively simple we are able to give sharp bounds for the diffusion time.  

\smallskip

In the last decades lots of effort has been put to obtain lower bounds of Sobolev norms for solutions of nonlinear Hamiltonian PDEs on compact manifolds. The first works in this direction are due to Bourgain \cite{Bou95},  \cite{Bou96} and Kuksin \cite{Kuksin97b}. In $2010$ Colliander-Keel-Staffilani-Takaoka-Tao (I-team) proved an outstanding result \cite{CKSTT} concerning the $H^s$-instability of the origin of the cubic nonlinear Schr\"odinger equation on $\mathbb{T}^2$ with $s>1$ \footnote{We remark that the energy provides a complete control of the $H^1$-norm.}. The construction of the unstable solutions is inspired by Arnold diffusion techniques: it relies on the presence of orbits of a finite dimensional good approximation of the equation (called \emph{toy model}) that shadow a chain of invariant hyperbolic manifolds (periodic orbits in a suitable symplectic reduction). After this work many papers have been devoted to the analysis and extensions of that scheme, exclusively for NLS models (concerning results on other models, using a different approach, we cite the seminal works by G\'erard-Grellier \cite{GerardG10}, \cite{GerardG11} on the Szeg\"o equation on the circle).
In \cite{GuardiaK12} Guardia-Kaloshin provide estimates for the diffusion time of solutions obtained by refining the analysis of the dynamics of the toy model,  Guardia-Haus-Procesi \cite{GuardiaHP16} extended the result of the I-team to Schr\"odinger equations with all types of analytic nonlinearities. The study of the $H^s$-instability of different invariant objects rather than elliptic fixed points have been carried out just recently. 
We mention the paper \cite{Hani12} by Hani for a proof of $H^s$-instability of plane waves and Guardia-Hani-Haus-Maspero-Procesi \cite{GuardiaHHMP19} for the case of finite gap solutions for $s\in (0, 1)$. 
We point out that the I-team mechanism strongly relies on the fact that the Schr\"odinger equation possess several symmetries. In particular the dynamics of the resonant models considered in the aforementioned papers is rather special and far from being generic \footnote{One can find arbitrarily long sequences of invariant manifolds on the same energy level. Moreover the heteroclinic connections between these manifolds are not transversal, in any possible sense.}. For this reason it seems hard to implement a similar strategy for different Hamiltonian PDEs.
Then it is natural to look for alternative mechanisms and we believe that a first step in this direction should be to find out different ways of transferring energy between modes, even restricting the study to single resonant clusters. About this, we mention the very recent result \cite{GGMP} by Guardia-Martin-Pasquali and the author of the present paper concerning chaotic-like transfers of energy for the wave, beam and Hartree equations on $\mathbb{T}^2$.

\section{Main results}\label{secmainres}
Let $s>0$, we introduce the Sobolev spaces
\[
H^s:=\left\{ u\in L^2(\mathbb{T}) : u=\sum_{j\in\mathbb{Z}} u_j\,e^{\mathrm{i} j x}\,,\,\,|| u ||_s^2:=\sum_{j\in\mathbb{Z}} | u_j |^2\,\langle j \rangle^{2 s} \right\} \]
where $\langle j \rangle:=\max\{ 1, |j| \}$.
We define the following norm {
\[
|| u ||_{E^s}:=|| u ||_{s+1/2}+|| u_t ||_{s-1/2}.
\]
}
\subsection*{Notations} When we say that a parameter $p$ is large (small) enough, or we write $p\gg1$ ($p\ll 1$), we mean that there exists a universal constant $p_0$ large (small enough) such that $p\geq p_0$ ($p\le p_0$).\\
The notation $p\lesssim q$ denotes that there exists a pure constant $C>0$ such that $p\le C q$.  The notation $p=\mathcal{O}(q)$ denotes that there are two pure constants $0<C_1< C_2$ such that $C_1 q\le p\le C_2 q$.

\medskip

The first result regards a class of nonlinear wave equations with even-power nonlinearities.

\begin{theorem}\label{thmwave}
{Let $s>2$}. Given $\delta\ll 1$ and $\mathcal{C}\gg 1$ there exists $p_0=p_0(s,\delta, \mathcal{C})$ such that for all \textbf{even} numbers $p\geq p_0$ the equation
\begin{equation}\label{wave}
u_{tt}-\Delta u+V_p*u+u^p=0,
\end{equation}
where
\begin{equation}\label{coeffpot}
V_p(x):=1+\cos(x)+p^2 \cos(p x),
\end{equation}
possesses a solution $u(t, x)$ such that $|| u(0) ||_{E^s}\le \delta$ and
\begin{equation}\label{factorgrowthC}
 \frac{|| u(T) ||_{E^s}}{|| u(0) ||_{E^s}}\geq \mathcal{C}
\end{equation}
with
\begin{equation}\label{timeteo}
T=\mathcal{O}( 2^p\,p^{-1}\,\,|| u(0) ||_{E^s}^{1- p}).
\end{equation}
\end{theorem}

The above theorem ensures the existence of a solution with arbitrarily small initial datum and arbitrarily large growth $\mathcal{C}$, but it does not guarantee that the norm at time $T$ is arbitrarily large. However this can be obtained for sufficiently higher order Sobolev norms.  

\begin{theorem}\label{stronginstability}
Given $\delta\ll 1$ and $K\gg 1$ there exists $p=p(\delta)$ such that the following holds.\\
There exists $s_0=s_0(p, K)$ such that the equation \ref{wave}
where $V_p$ as in \ref{coeffpot},
possesses a solution $u(t, x)$ such that 
\begin{equation}\label{strongb}
|| u(0) ||_{E^s}\le \delta, \qquad || u(T) ||_{E^s}\geq K \qquad \forall s\geq s_0,
\end{equation}
with $T$ as in \ref{timeteo}. Moreover if $K=\delta^{-\alpha}$ with $\alpha>1$ we can provide the following polynomial time estimate with respect to the growth
\begin{equation}\label{timeC}
T\lesssim \mathcal{C}^{p-1}, \qquad \mathcal{C}:=\frac{K}{\delta}.
\end{equation}

\end{theorem}

\noindent Some comments are in order:
\begin{itemize}
\item  The nonlinearity in equation \ref{wave} can be replaced by any analytic function $f(u)$ with a zero of order $p$ at the origin $u=0$. Moreover the sign of the nonlinearity does not play any role.\\
The degree $p$ is used as a parameter to deal with increasingly high order resonances in the first step of a Birkhoff normal form procedure. More precisely, it turns out that the modes $\pm 1$ (low modes) and $\pm p$ (high modes) are in resonance.
\item The degeneracy of the resonant Hamiltonian of the nonlinear term is due to the evenness of the degree $p$.  An evidence of the unstable behavior of wave equations with even-power nonlinearities on $\mathbb{T}$ is given, for instance, by the result of non-existence of periodic solutions given in \cite{BertiPro}. 
\item Although the results are rather different, it is interesting to notice that the diffusion time for arbitrarily small initial data solutions obtained through the I-team mechanism is super-exponential with respect to the growth, $\exp(\mathcal{C}^{k})$ for some $k>1$ (see for instance \cite{GuardiaHP16}), while \ref{timeC} is a polynomial bound.    
\item We remark that the growth of $H^s$-Sobolev norms with $s\geq s_0$ for small data solutions obtained in Theorem \ref{stronginstability} is not trivial, especially because the diffusion time does not increase with the index $s$.\\ 
A result of this kind can be achieved if there is a transfer of energy and the norm of the initial datum remains small when $s$ increases. For the case taken into account this happens because: (i) the Fourier support of the initial datum includes the modes $\pm 1$, whose Sobolev weights are the same for all $s$, and (ii) the time of diffusion has lower and upper bounds that do not depend on the initial amount of energy of the high modes $\pm p$. The latter fact is due to the mechanism we are considering, which relies on the existence of diffusion channels (see Remark \ref{remarkfast}). 
\end{itemize}

The next theorem concerns cubic NLS equations which are not $x$-translation invariant. A similar model has been considered in \cite{GrebertV11} for the search of time-periodic beating solutions.\begin{theorem}\label{thmnls}
Let $s>0$. Given $\delta\ll 1$ and $K\gg 1$ there exists $N_0=N_0(s, \delta, K)$ such that 
 for all $N\geq N_0$ there exists a trigonometric polynomial $V_N\colon \mathbb{T}\to \mathbb{R}$ with real Fourier coefficients supported on $4$ modes associated to wavenumbers $k_1, k_2, k_3$ and $k_4$, where $| k_1 |, | k_2 |, | k_3|\lesssim \sqrt{N}$ and $|k_4|=\mathcal{O}(N)$, such that the following holds:\\
 The equation
\begin{equation}\label{nls}
\mathrm{i} u_{t}-\Delta u+V_N* u+ \cos(N x) |u|^2 u=0
\end{equation}
possesses a solution $u(t, x)$
such that
\begin{equation*}\label{growthwave}
|| u(0) ||_{s}\le \delta, \quad || u(T) ||_{s}\geq K
\end{equation*}
with
\begin{equation}\label{timecondiniziale}
T=\mathcal{O}(N^{ s}\,\,|| u(0) ||_{s}^{-2}).
\end{equation}
Moreover if $K=\delta^{-\alpha}$ for some $\alpha>1$ we can provide the following polynomial time estimate with respect to the growth 
\begin{equation}\label{timeC2}
T\lesssim \mathcal{C}^6, \qquad \mathcal{C}:=\frac{K}{\delta}.
\end{equation}
\end{theorem}

\noindent Some comments are in order:

\begin{itemize}
\item We remark that in dimension one the cubic NLS is completely integrable and, by the presence of infinitely many constants of motion, the Sobolev norms are controlled for all time. In the equation \ref{nls} the Hamiltonian structure and the mass (or the $L^2$-norm) are still preserved, but the nonlinear term breaks the momentum conservation. The frequency of the cosine $x$-function is used as a parameter to involve modes of very different size scale in the $4$-resonant interactions. Broadly speaking, the ratio between the size of the low and high modes turns out to be a power of $N$.
\item We remark that when $V_N=0$ the $H^1$-norm of solutions with small $L^2$-norm is controlled for all time. This can be seen by using that $\cos(N x)$ is uniformly bounded in $N$ and by applying the Gagliardo-Nirenberg inequality. When $V_N\neq 0$ the $H^1$-norm has still an upper bound for all time, but it is not uniform in $N$.
\end{itemize}

The convolution potentials $V_p$ in \ref{wave} and $V_N$ in \ref{nls} have the role to confine the dynamics of the normalized Hamiltonian on a finite dimensional submanifold, but they still preserve the resonant nature of the equations.
The solutions that we construct bifurcate from a periodic or quasi-periodic in time function $w(t, x)$ (see \ref{bif}, \ref{bif2}) that is obtained as a solution of the linearized problem at the origin by exciting a finite number of modes. The orbit $w(\cdot, x)$ fills densely a lower dimensional submanifold of an embedded resonant torus. The solutions $u(t, x)$ provided by the above theorems remain close to $w(t, x)$ in a weak norm for $t\in [0, T]$, but clearly not in the $H^s$-topology.

\medskip

The diffusion channels that we exploit are contained in the level sets of quadratic constants of motion, respectively the momentum and the mass for the wave equation \ref{wave} and the Schr\"odinger equation \ref{nls} (see \ref{fix} and \ref{massls}). The unstable solutions that we obtained come in one-parameter families, parametrized by the values of the aforementioned first integrals.

\subsubsection*{Comparison with beating solutions}\label{comparisonbeating}
To optimize the ratio between the Sobolev norms at time $t=0$ and at some other time $t=T$ we want to set the initial energy of the high modes almost at zero, say $\varepsilon$-small.
Thanks to the use of diffusion channels the exchanging time turns out to have an upper bound independent of $\varepsilon$ (see for example Lemma \ref{sistdin} and Remarks \ref{remarkfast}, \ref{rmkepsindip}). This is not the case if the same amount of energy is transferred among the modes of a periodic beating solution. These solutions
are usually obtained following periodic orbits of a resonant model which are very close to heteroclinic or homoclinic loops. Let us suppose for simplicity that the transfer of energy involves just two modes. Setting the initial energy of one of the modes almost at zero corresponds to consider an orbit with very large period that visits a small neighborhood $\mathcal{U}$ of a saddle point (or a hyperbolic manifold). If the size of $\mathcal{U}$ is $\mathcal{O}(\varepsilon)$ then the time spent to escape from it is $\mathcal{O}(\log(\varepsilon))$.\\ If one wants to prove that this kind of transfers of energy produces a growth of $H^s$-Sobolev norms for $s\gg 1$ then $\varepsilon$ has to be chosen such that the norm at initial time does not increase with $s$. Then $\varepsilon\sim C^{-s}$ for some constant $C>0$ (see for instance \ref{epschoice} and \ref{parameters2}) and the factor $\log(\varepsilon)$ makes the diffusion time explodes as $s\to \infty$.

\subsection{Scheme of the proofs} The proofs of Theorems \ref{thmwave}, \ref{thmnls} follow the same steps. Since the result on the nonlinear wave equations is more complicated we give full details for the proof of Theorem \ref{thmwave} and provide an outline of the proof for Theorem \ref{thmnls}.\\
Let us briefly describe the general strategy. First we introduce the Hamiltonian structure of equations \ref{wave} and \ref{nls}. Then we perform a Birkhoff normal form algorithm (see Propositions \ref{wbnf}, \ref{wbnfnls}), namely we provide a change of coordinates that remove some terms from the Hamiltonian that do not affect the dynamics in a neighborhood of the origin for a certain range of times. 
 The Hamiltonian in normal form turns out to possess finite-dimensional invariant subspaces. The restriction of the normalized system to such spaces defines our \emph{resonant model}. We analyze the dynamics of the resonant model by using action-angle variables \footnote{We refer to \cite{BMP} for the analysis of analogous finite-dimensional models.}. We construct an orbit that exhibit a certain drift in the actions in the optimal diffusion time (see Lemmata \ref{sistdin}, \ref{sistdinnls}). Thanks to a rescaling and Gronwall arguments we prove that there exists a solution of the full PDE that remains close (in a weak norm) to the unstable orbit of the resonant model for sufficiently long time (see Propositions \ref{approxarg}, \ref{approxarg2}). The last step consists in the proof of the bounds for the Sobolev norms of the solution at time $t=0$ and $t=T$, where $T$ is the rescaled diffusion time.

\section*{Acknowledgments} 
The idea of applying the theory of fast instability in the context of Hamiltonian PDEs has been suggested by V. Kaloshin to L. Biasco and M. Procesi some time ago. Recently L. B. and M. P. told me about that discussion and suggested me to read the papers \cite{BCT2006}, \cite{KalBoune} about diffusion channels. This has inspired the present work. So I would like to thank them all.\\
I am grateful also to M. Guardia, R. Scandone, E. Haus and R. Feola for many useful comments and discussions.\\
This paper has received funding from 
the European Research Council (ERC) 
under the European Union's Horizon 2020
research and innovation programme under grant agreement 
No 757802.

\section{Proof of Theorem \ref{thmwave}}\label{sec:wave}

\subsection{Hamiltonian structure}
In this section we introduce a useful set of coordinates and we discuss the Hamiltonian structure of the equation \ref{wave}. To simplify the notation we drop the subindex $p$ from the potential, namely we write $V_p=V$.
Let us denote by $\Lambda:=-\Delta+V*$. 
The wave equation \ref{wave} can be written as a first order system
\begin{equation*}\label{systemwave}
\begin{cases}
\dot{u}=v,\\
\dot{v}=-\Lambda u-u^{p},
\end{cases}
\end{equation*}
that, in the following complex coordinates {
\begin{equation}\label{complexcoord}
z^+:=\frac{\Lambda^{1/4}u-\mathrm{i} \Lambda^{-1/4} v}{\sqrt{2}}, \quad z^-:=\frac{\Lambda^{1/4}u+\mathrm{i} \Lambda^{-1/4} v}{\sqrt{2}},
\end{equation}
}
reads as
\begin{equation}\label{systemcomplex}
\begin{cases}
\mathrm{i} \dot{z}^+= -\Lambda^{1/2}  z^+-\mathtt{g}(z^+, z^-),\\
\mathrm{i} \dot{z}^-= \Lambda^{1/2}  z^-+\mathtt{g}(z^+, z^-),
\end{cases}
\end{equation}
with {
\[
\mathtt{g}(z^+, z^-):=\frac{1}{\sqrt{2}} \Lambda^{-1/4} \left( \Lambda^{-1/4} \left( \frac{z^++z^-}{\sqrt{2}} \right) \right)^{p}.
\]
}
\begin{remark}
We provide a solution of \ref{systemcomplex} that undergoes growth in the high Sobolev norms $|| \cdot ||_s$. Then it is easy to recover the same result for the norms $|| \cdot ||_{E^s}$ by undoing the change of coordinates \ref{complexcoord}.
\end{remark}
We introduce the infinitely many coordinates
\begin{equation*}\label{fouriercoord}
z^+=\sum_{j\in\mathbb{Z}} z_j^+ e^{\mathrm{i} j x}, \quad z^-=\sum_{j\in\mathbb{Z}} z_j^- e^{-\mathrm{i} j x}
\end{equation*}
that transform the system \ref{systemcomplex} in an infinite dimensional system of ODEs in the unknowns $(z_j^+, z_j^-)$, $j\in\mathbb{Z}$.
Such system is equipped with a Hamiltonian structure given by the symplectic form $-\mathrm{i}\sum_{j\in\mathbb{Z}} d z_j^+\wedge d z_j^-$, which in turn induces the Poisson structure
\begin{equation}\label{poisson}
\{ F, G\}:=-\mathrm{i} \sum_{j\in\mathbb{Z}} \big(\partial_{z_j^+} F\,\partial_{z_j^-} G-\partial_{z_j^-} F\,\partial_{z_j^+} G\big),
\end{equation}
where $F$ and $G$ are two real-valued functions defined on the phase space.
 The Hamiltonian of \ref{systemcomplex} is given by
\begin{equation}\label{hamwave}
\begin{aligned}
H(z_j^+, z_j^-)=&\sum_{j\in\mathbb{Z}} \omega(j) z_j^+ z_j^-\\
&{+ \frac{1}{(p+1)\sqrt{2}^{(p+1)}} \sum_{j_1+\dots+j_{p+1}=0} \frac{(z^+_{j_1}+z_{-j_1}^-)\dots (z^+_{j_{p+1}}+z^-_{j_{-(p+1)}})}{\sqrt{\omega(j_1)\dots \omega(j_{p+1})}},}
\end{aligned}
\end{equation}
where
\[
\omega(j):=\sqrt{j^2+V_j}, \quad j\in\mathbb{Z}
\]
are the linear frequencies of oscillation.
We shall look for a solution mainly Fourier supported on the symmetric \emph{tangential} set
\begin{equation*}\label{tangentialset}
S:=S^{+}\cup S^{-}, \quad S^{\pm}:=\{ \pm 1, \pm p \}.
\end{equation*}
By the choice of the convolution potential as in \ref{coeffpot} the linear frequencies of oscillation are given by
\begin{equation}\label{linearfreq}
\omega(j):=\begin{cases}
1 \quad \qquad \quad \quad \quad j=0,\\
\sqrt{2} |j| \qquad \qquad j\in S,\\
|j| \quad \quad \qquad j\notin S\cup\{0\}.
\end{cases}
\end{equation}
\begin{remark}
We could choose $V_0=n$ with an integer $n\geq 1$. Moreover $\sqrt{2}$ can be replaced by any badly approximable number. Indeed all we need is to use the fact that
\begin{equation}\label{irrational}
{|\sqrt{2} n+m|\geq \frac{\gamma}{n}} \quad \forall n, m\in \mathbb{Z}, \quad n \neq 0
\end{equation}
for some $\gamma\in (0, 1)$.
\end{remark}
\begin{remark}\label{rmkdecouple}
The frequencies $\omega(j)$ with $j\in S$ (the \emph{tangential frequencies}) are irrational, while the normal frequencies are integer numbers. This is the key property that allows to decouple the resonant dynamics of the tangential and normal modes.
\end{remark}

We point out that the real subspace
\begin{equation}\label{realsub}
\mathtt{R}:=\left\{ \overline{z_j^+}=z_j^- \right\}
\end{equation}
is invariant under the flow of $H$. Since we shall work on $\mathtt{R}$ it is convenient to adopt the following notation
\[
z_j:=z_j^+, \quad \overline{z_j}:=z_j^-.
\]
We observe that by exciting the tangential modes $\{\pm 1, \pm p\}$ we obtain a solution $w(t, x)$ of the linearized problem at the origin 
\[
\mathrm{i} \dot{z}_j=\omega(j) \,z_j, \quad j\in\mathbb{Z},
\] of the form
\begin{equation}\label{bif}
w(t, x)=a_1\,e^{\mathrm{i} \sqrt{2} t}\, \cos(x)+a_p\,e^{\mathrm{i} \sqrt{2} p t}\, \cos(p x), \quad a_1, a_p\in\mathbb{C}.
\end{equation}
These linear solutions can be seen as periodic motions supported on invariant embedded tori of dimension $4$. We expect that even a small perturbation, which here is provided by the nonlinear terms, may destroy these resonant manifolds and give rise to diffusive orbits.

\smallskip

We write the Hamiltonian \ref{hamwave} as $H=H^{(2)}+H^{(p+1)}$ where
\begin{equation}\label{ham}
\begin{aligned}
H^{(2)}(z, \overline{z}):=&\sum_{j\in\mathbb{Z}} \omega(j) z_j \overline{z_j},\\
H^{(p+1)}(z, \overline{z}):= &\frac{1}{(p+1)\sqrt{2}^{(p+1)}}\sum_{\substack{ (\alpha, \beta)\in{\mathbb{N}}^{\mathbb{Z}}\times {\mathbb{N}}^{\mathbb{Z}},\\ |\alpha|+|\beta|=p+1 ,\\\pi(\alpha, \beta)=0}} C_{\alpha, \beta}\, z^{\alpha}\,\overline{z}^{\beta}
\end{aligned}
\end{equation}
with $| \alpha|=\sum_{j\in\mathbb{Z}} \alpha_j$, $\pi(\alpha, \beta):=\sum_{j\in\mathbb{Z}} j \,(\alpha_j-\beta_j)$,
\[
z^{\alpha}:=\prod_{j\in\mathbb{Z}} z_j^{\alpha_j}, \quad \overline{z}^{\beta}:=\prod_{j\in\mathbb{Z}} \overline{z_j}^{\beta_j}
\]
and the coefficients
\begin{equation}\label{coeff}
C_{\alpha, \beta}:=\frac{(p+1)!}{\alpha! \beta!}\, \prod_{j\in\mathbb{Z}} \omega(j)^{-\frac{\alpha_j+\beta_j}{2}}\in\mathbb{R}.
\end{equation}
\begin{remark}
We observe that a monomial $z^{\alpha}\,\overline{z}^{\beta}$ commutes with the momentum Hamiltonian
\begin{equation}\label{momentum}
M(z, \overline{z}):=-\mathrm{i} \sum_{j\in \mathbb{Z}} j |z_j|^2
\end{equation}
if and only if $\pi(\alpha, \beta)=0$.
\end{remark}

The vector field of $H$ is defined by components as
\[
X_H:=\left(X_H^{(z_j)}, X_H^{(\overline{z}_j)} \right), \quad X_H^{(z_j)}:=\mathrm{i} \partial_{\overline{z_j}} H, \quad X_H^{(\overline{z_j})}:=-\mathrm{i} \partial_{z_j} H, \quad j\in\mathbb{Z}.
\]

\subsection{Birkhoff normal form}\label{sec:bnf}
In this section we perform a Birkhoff normal form procedure in order to highlight the terms of the Hamiltonian \ref{ham} which give the effective dynamics of equation \ref{wave} for a certain range of time.
We shall work on the space of sequences
\begin{equation}\label{elle1}
\ell^1:=\left\{ z\colon\mathbb{Z}\to \mathbb{C}\, | \, || z ||_{\ell^1}:=\sum_{j\in\mathbb{Z}} | z_j |<\infty \right\}.
\end{equation}
We point out that $\ell^1$ is an algebra with respect to the convolution product.\\ We denote by $B_{\eta}$ the ball centered at the origin of $\ell^1$ with radius $\eta>0$, namely
\[
B_{\eta}:=\{ z\in \ell^1 : || z||_{\ell^1}\le \eta \}.
\]

\begin{definition}\label{def}
Let \[
F=F(z, \overline{z}):=\sum_{\substack{(\alpha, \beta)\in{\mathbb{N}}^{\mathbb{Z}}\times {\mathbb{N}}^{\mathbb{Z}},\\ |\alpha|+|\beta|=q ,\\\pi(\alpha, \beta)=0}} F_{\alpha, \beta}\, z^{\alpha}\, \overline{z}^{\beta}
\] be a homogeneous Hamiltonian of degree $q\geq 2$ preserving momentum. We give the following definitions:\\
\noindent (i) Let $0\le k\le q$, we denote by $F^{(q, k)}$ the projection of $F$ on the monomials supported on
\[
\mathcal{A}_{q, k}:=\left\{ (\alpha, \beta)\in{\mathbb{N}}^{\mathbb{Z}}\times {\mathbb{N}}^{\mathbb{Z}} : \,\pi(\alpha, \beta)=0,\,\,|\alpha|+|\beta|=q,\,\, \#(\alpha, \beta, S^c)=k \right\},
\]
where {$\#(\alpha, \beta, S^c):=\sum_{j\notin S} (|\alpha_j|+|\beta_j|)$} .\\ Similarly we define $F^{(q, \le k)}$, $F^{(q, \geq k)}$ as the projection of $F$ on the monomials supported respectively on
\[
\mathcal{A}_{q, \le k}=\bigcup_{i=1,\dots, k} \mathcal{A}_{q, i}, \quad \mathcal{A}_{q, \geq k}=\bigcup_{i=k,\dots, q} \mathcal{A}_{q, i}.\\
\]
\noindent(ii) We define the following norms
\[
[ \! [ F  ] \! ]:=\sup_{(\alpha, \beta)} |F_{\alpha, \beta}|,
\qquad  || X_F ||_{\eta}:=\sum_{j\in\mathbb{Z}} \sup_{B_{\eta}}  \left|X_F^{(z_j)} \right|.
\]
\end{definition}
\begin{lemma}\label{lemmaGT}
Let $F, G$ be two homogeneous Hamiltonians preserving momentum of degree $q$ and $\tilde{q}$ respectively. The Poisson bracket $\{ F, G\}$ defined in \ref{poisson} is a homogeneous Hamiltonian preserving momentum of degree $q+\tilde{q}-2$.\\
Moreover we have the following estimates
\begin{align}
| F(z, \overline{z}) |\le&\,  [ \! [ F ] \! ] || z ||_{\ell^1}^{q},\\ \label{lemvec}
|| X_F(z, \overline{z}) ||_{\ell^1}\le&  \,q\, [ \! [ F ] \! ] || z ||_{\ell^1}^{q-1}, \\ \label{lempoiss}
[ \! [ \{ F, G\} ] \! ]\le&  q\, \tilde{q}\,[ \! [ F ] \! ] [ \! [ G ] \! ].
\end{align}
\end{lemma}
\begin{proof}
The proof follows the same lines of the proof of Lemma $3.2$ in \cite{GT}.
\end{proof}
We denote by $\Pi_{\mathrm{Ker}}$ and $\Pi_{\mathrm{Rg}}$ the projection on the kernel and the range, respectively, of the adjoint action of $H^{(2)}$
\[
\mathrm{ad}_{H^{(2)}}[F]:=\{F, H^{(2)}\}.
\]
The adjoint action of $H^{(2)}$ is diagonal on the monomials $z^{\alpha}\,\overline{z}^{\beta}$ with eigenvalues {$-\mathrm{i} \Omega(\alpha, \beta)$, where}
\[
\Omega(\alpha, \beta):=\sum_{j\in\mathbb{Z}} \omega(j) (\alpha_j-\beta_j).
\]

\begin{proposition}{(\textbf{Birkhoff normal form})}\label{wbnf}
Recall \ref{ham}. There exists $\eta>0$ small enough such that there exists a symplectic change of coordinates $\Gamma\colon B_{\eta}\to B_{2\eta}$ which takes the Hamiltonian $H$ into its (partial) Birkhoff normal form up to order $p+1$, namely
\begin{equation}\label{hamBNFcoord}
H\circ \Gamma=H^{(2)}+H_{\mathrm{res}}+{H^{(p+1, \geq 2)}}+R
\end{equation}
where:\\
(i) The resonant Hamiltonian is given by
\begin{equation}\label{resham}
H_{\mathrm{res}}:=\Pi_{\mathrm{Ker}} H^{(p+1, 0)}=\frac{1}{\sqrt{2}^{p+1}} \big(2\Re(z_1^p\,\overline{z_{p}}\,)+2\Re(z_{-1}^p\,\overline{z_{-p}}\,) \big).
\end{equation}
(ii) The remainder $R$ is such that
\begin{equation}\label{boundR}
|| X_R ||_{\eta}\le C_1 \gamma^{-1} \eta^{2p-1}+\widetilde{C_1} \gamma^{-2} \eta^{3p-2}
\end{equation}
with
\begin{equation}\label{C1tildeC1}
C_1= 2^{p-1} p^{3} ((p+1)!)^2, \quad \widetilde{C_1}=(3p-1)\,p^{5} 2^{\frac{3}{2} p-\frac{5}{2}}\, ((p+1)!)^3.
\end{equation}
Moreover the map $\Gamma$ is invertible and close to the identity
\begin{equation}\label{gammaid}
|| \Gamma^{\pm 1}-\mathrm{Id} ||_{\eta}\le C_0\, \gamma^{-1} \,\eta^{p}
\end{equation}
with
\[
C_0=\sqrt{2}^{p-1} p^2 (p+1)!
\]
and it preserves the real subspace $\mathtt{R}$ in \ref{realsub}.
\end{proposition}
\begin{remark}
With a slight abuse of notation we have renamed the Birkhoff coordinates $(z_j)_j$ as the original ones. 
\end{remark}

\begin{proof}
We consider the following homogeneous Hamiltonian
\[
F=\sum_{{\mathcal{A}_{p+1, \le 1}}} F_{\alpha, \beta} z^{\alpha}\,\overline{z}^{\beta}
\]
with
\[
F_{\alpha, \beta}:=\begin{cases}
\dfrac{{\mathrm{i}}\,C_{\alpha, \beta}\, }{\Omega(\alpha, \beta) \sqrt{2}^{p+1} (p+1)} \qquad \text{if}\,\,\Omega(\alpha, \beta)\neq 0,\\
0 \qquad \qquad \qquad \qquad \qquad \quad \text{otherwise.}
\end{cases}
\]
The coefficients $F_{\alpha, \beta}$ are uniformly bounded since $\mathcal{A}_{p+1, \le 1}$ is a finite set. Thus the Hamiltonian $F$ is well defined and by using Young's inequality is easy to prove that the associated equation is locally well posed in $\ell^1$. By a standard bootstrap argument one can prove that, provided $\eta$ is small enough, the flow $\Phi^t_F$ maps $B_{\eta}$ into $B_{2\eta}$ for $t\in [0, 1]$. We define $\Gamma:=\Phi^1_F$.\\
By the definition of $F_{\alpha, \beta}$ and the fact that $\omega(j)\in\mathbb{R}$ for all $j\in\mathbb{Z}$ (see \ref{linearfreq}) the vector field $X_F$ preserves the real subspace $\mathtt{R}$ in \ref{realsub}.\\
By definition $F$ satisfies the following homological equation
\begin{equation}\label{homoeq}
\{F, H^{(2)}\}+{H^{(p+1, \le 1)} =\Pi_{\mathrm{Ker}} H^{(p+1, \le 1)}.}
\end{equation}
We claim that 
\begin{equation}\label{claim}
\Pi_{\mathrm{Ker}} H^{(p+1, 1)}=0. 
\end{equation} 
If $(\alpha, \beta)\in \mathcal{A}_{p+1, 1}$ then there exists $j\notin S$ such that (note that $\omega(j)=\omega(-j)$ for all $j\in\mathbb{Z}$)
\[
\Omega(\alpha, \beta)=\omega(1) (\alpha_1+\alpha_{-1}-\beta_1-\beta_{-1})+\omega(p) (\alpha_{p}+\alpha_{-p}-\beta_p-\beta_{-p})+ |j| (\alpha_j-\beta_j)
\]
with $(\alpha_j, \beta_j)=(1, 0)$ or $(\alpha_j, \beta_j)=(0, 1)$ and $\sum_{i\in S} \alpha_i+\beta_i=p$. 
Then by \ref{linearfreq} we have that $\Omega(\alpha, \beta)=\sqrt{2} n+m$ for some $n, m\in\mathbb{Z}$ with $0<|n |\le {p^2}$, hence (see \ref{irrational})
\begin{equation}\label{nosmalldiv}
|\Omega(\alpha, \beta)|\geq \frac{\gamma}{p^2}>0.
\end{equation}

This proves the claim \ref{claim}. Now we prove the estimate \ref{gammaid} on the map $\Gamma$.
 By \ref{coeff}, \ref{linearfreq} we have
\begin{equation}\label{boundcoeff}
|C_{\alpha, \beta}|\le (p+1)!.
\end{equation}
If $(\alpha, \beta)\in \mathcal{A}_{p+1, 0}$ and $\Omega(\alpha, \beta)\neq 0$ then by the definition of the tangential frequencies in \ref{linearfreq}
\[
|\Omega(\alpha, \beta)|\geq \sqrt{2}.
\]
Hence by Lemma \ref{lemmaGT}-\ref{lemvec} we have 
\begin{equation}\label{effe}
[ \! [ F ] \! ]\le  \gamma^{-1} \sqrt{2}^{-(p+1)} p^{2} p! , \qquad || X_F ||_{\eta}\le \gamma^{-1} \sqrt{2}^{-(p+1)} p^{2} (p+1)! \eta^{p}.
\end{equation}
By the mean value theorem and using that $\Phi^t_F\colon B_{\eta}\to B_{2\eta}$ for $t\in [0, 1]$ we have
\[
|| \Gamma(z)-z ||_{\eta}\le \sup_{t\in [0, 1]} || X_F(\Phi^t_F(z)) ||_{\eta}\le \gamma^{-1} \sqrt{2}^{p-1} p^{2} (p+1)! \eta^{p}.
\]
This gives the bound \ref{gammaid} for $\Gamma$.
If $\eta$ is small enough we can invert $\Gamma$ by Neumann series and get a similar bound for the inverse.\\
Now we prove formula \ref{resham}.
By \ref{claim} we focus on monomials of the following form
\[
\prod_{i\in S} z_i^{\alpha_i}\,\overline{z_{i}}^{\beta_i}
\]
where
\begin{equation}\label{sum}
\sum_{i\in S}\alpha_i+\sum_{i\in S} \beta_i=p+1, \qquad \alpha_i, \beta_i\geq 0.
\end{equation}
These monomials are resonant if (recall \ref{linearfreq})
\begin{equation}\label{res}
\big(\alpha_1+\alpha_{-1}-\beta_{1}-\beta_{-1}\big)+p (\alpha_{p}+\alpha_{-p}-\beta_{p}-\beta_{-p})=0.
\end{equation}
While
the momentum conservation reads as
\begin{equation}\label{mom}
\big(\alpha_1-\alpha_{-1}-\beta_{1}+\beta_{-1}\big)+p (\alpha_{p}-\alpha_{-p}-\beta_{p}+\beta_{-p})=0.
\end{equation}
First we observe that \ref{res} implies
\begin{equation}\label{modulo}
p |\alpha_{p}+\alpha_{-p}-\beta_{p}-\beta_{-p}|= |\alpha_1+\alpha_{-1}-\beta_{1}-\beta_{-1}|.
\end{equation}
If $|\alpha_{p}+\alpha_{-p}-\beta_{p}-\beta_{-p}|\neq 0$ then by \ref{sum} we have $p+1\geq |\alpha_1+\alpha_{-1}-\beta_{1}-\beta_{-1}|\geq p$.\\
The case $|\alpha_1+\alpha_{-1}-\beta_{1}-\beta_{-1}|=p+1$ is clearly impossible since the left hand side of \ref{modulo} is even.
Thus 
  \begin{equation}\label{uno}
  |\alpha_1+\alpha_{-1}-\beta_{1}-\beta_{-1}|=p, \quad |\alpha_{p}+\alpha_{-p}-\beta_{p}-\beta_{-p}|=1.
  \end{equation}
By the latter equality we deduce that there is exactly one integer number in $\{\alpha_{\pm p}, \beta_{\pm p}\}$ which is equal to $1$, while all the others are zero. Then \ref{mom} implies that
   \begin{equation}\label{due}
   |\alpha_1-\alpha_{-1}-\beta_{1}+\beta_{-1}|=p.
   \end{equation}
    It is easy to see that \ref{uno} and \ref{due} imply $\alpha_1=\beta_{1}$ or $\alpha_{-1}=\beta_{-1}$. Without loss of generality suppose that $\alpha_1=\beta_{1}$, then $|\alpha_{-1}-\beta_{-1}|=p$. This means that $(\alpha_{-1}, \beta_{-1})=(p, 0)$ or $(\alpha_{-1}, \beta_{-1})=(0, p)$ (and so by \ref{sum} $\alpha_1=\beta_{1}=0$).\\ The resonant monomials corresponding to these cases are
  \[
  z_{p} \overline{z_{1}}^{p}, \quad z_{-p } \overline{z_{-1}}^{p} 
  \]
  and their complex conjugate .
  We are left with the case $|\alpha_{p}+\alpha_{-p}-\beta_{p}-\beta_{-p}|=0$. By \ref{modulo} we have 
  $$|\alpha_1+\alpha_{-1}-\beta_{1}-\beta_{-1}|=0.$$ By the fact that $\alpha_i, \beta_i\geq 0$ this implies that
  \[
  \alpha_j+\alpha_{-j}=\beta_j+\beta_{-j}, \qquad j=1, p .
  \]
  Then \ref{sum} becomes $2(\alpha_1+\alpha_{-1}+\alpha_{p}+\alpha_{-p})=p+1$, which is a contradiction since $p+1$ is odd. \\
 Now we prove \ref{hamBNFcoord}. The new Hamiltonian is obtained by Taylor expanding $H\circ \Phi^t_F$ at $t=0$. We have 
 \begin{align*}
 H\circ \Gamma&=H+\{ F, H\}+\frac{1}{2}\int_0^1(1-t) \{ F, \{ F, H\}\}\circ \Phi^t_F\,dt\\
 &=H^{(2)}+\{  F, H^{(2)}\}+{H^{(p+1, \le 1)}+H^{(p+1, \geq 2)}}+\{F, H^{(p+1)}\}\\
 &+\frac{1}{2}\int_0^1(1-t) \{ F, \{ F, H^{(2)}\}\}\circ \Phi^t_F\,dt+\frac{1}{2}\int_0^1(1-t) \{ F, \{ F, H^{(p+1)}\}\}\circ \Phi^t_F\,dt\\
 &\stackrel{\ref{homoeq}, \ref{claim}}{=} H^{(2)}+\Pi_{\mathrm{Ker}} H^{(p+1, 0)}+{H^{(p+1, \geq 2)}}+R
 \end{align*}
 where
 \begin{align*}
 R:=&\{F, H^{(p+1)}\}-\frac{1}{2}\int_0^1(1-t) \{ F, \Pi_{\mathrm{Rg}} {H^{(p+1, \le 1)}}\}\circ \Phi^t_F\,dt\\
 &+\frac{1}{2}\int_0^1(1-t) \{ F, \{ F, H^{(p+1)}\}\}\circ \Phi^t_F\,dt.
 \end{align*}
 By using Lemma \ref{lemmaGT}-\ref{lempoiss} and the bounds \ref{effe}, \ref{boundcoeff} we obtain the following estimates
 \[
 [ \! [ H^{(p+1)} ] \! ]\le p! \sqrt{2}^{-(p+1)},
 \]
 \[
 [ \! [ \{ F, H^{(p+1)}\} ] \! ]\le  \gamma^{-1} 2^{-(p+1)} p^{2} ((p+1)!)^2   ,
 \]
 \[
  [ \! [ \{F, \{ F, H^{(p+1)}\}\} ] \! ]\le  \, \gamma^{-2} p^{5} 2^{-\frac{3}{2} p-\frac{1}{2}}\, ((p+1)!)^3 .  
 \]
 Then using again Lemma \ref{lemmaGT}-\ref{lemvec} we get the estimate \ref{boundR}.
\end{proof}
\subsection{Dynamics of the resonant model}
We introduce the rotating coordinates
\[
z_j=r_j \,e^{\mathrm{i} \omega(j) t}
\]
in order to remove the quadratic part of the Hamiltonian $H\circ \Gamma$. {If $z$ is a solution of \ref{hamBNFcoord}} then $r$ satisfies the equation associated to the Hamiltonian 
\begin{equation}\label{hamtime}
\mathcal{H}=H_{\mathrm{res}}+\mathcal{Q}(t)+\mathcal{R}(t)
\end{equation}
where
\begin{equation}\label{hamrotating}
\begin{aligned}
&{\mathcal{Q}\big((r_j)_{j\in\mathbb{Z}}, t\big):=H^{(p+1, \geq 2)}\big(r_j \,e^{\mathrm{i} \omega(j) t}\big)},\\
& \mathcal{R}\big((r_j)_{j\in\mathbb{Z}}, t\big):=R\big(r_j \,e^{\mathrm{i} \omega(j) t}\big).
\end{aligned}
\end{equation}
The next step is to study the dynamics of the resonant Hamiltonian $H_{\mathrm{res}}$.\\ We observe that the finite dimensional submanifold
\[
\mathcal{U}_S:=\{ r\colon \mathbb{Z}\to\mathbb{C} \,|\, r_j=0\,\,\,j\notin S\}
\]
is invariant by the flow of $H_{\mathrm{res}}$. We introduce the following action-angle variables on $\mathcal{U}_S$
\[
r_j=\sqrt{I_j}\,e^{\mathrm{i} \theta_j} \quad j\in S.
\]
The reduced Hamiltonian now reads as (recall \ref{resham})
\begin{equation}\label{hamred}
\mathcal{G}:=\mathcal{G}^++\mathcal{G}^-, \quad \mathcal{G}^{\pm}= \sqrt{2}^{1-p}\, I_{\pm 1}^{p/2}\,\sqrt{I_{\pm p}} \,\cos(p\, \theta_{\pm 1}-\theta_{\pm p}).
\end{equation}
We observe that $\mathcal{G}$ is the sum of two uncoupled integrable Hamiltonians, indeed both $\mathcal{G}^{\pm}$ have one degree of freedom in an opportune reduction. Then it makes sense to analyze just $\mathcal{G}^+$. 
\begin{remark}
 The partial momenta $\mathtt{M}_{\pm}:=\pm I_{\pm 1}\pm p I_{\pm p}$ are constants of motion for $\mathcal{G}_{\pm}$ respectively. 
\end{remark}
The following lemma provides the existence of an orbit that exhibit a large drift in one of its actions.
\begin{lemma}\label{sistdin}
Let $\varepsilon>0$ be arbitrarily small and $\mathtt{c}> p\varepsilon$. There exists an orbit of $\mathcal{G}^{+}$
\[
g^+_{\varepsilon, \mathtt{c}}(t)=(\theta_1(t), \theta_{p}(t),  I_1(t),  I_{p}(t)) 
\]
 such that
 \begin{equation} \label{ic}
\begin{aligned}
I_1(0) &= \mathtt{c}-p\,\varepsilon, &  I_{p}(0)&=\varepsilon, \\
 I_1(T_0) &=\mathtt{c}(1-p^{-1}), &  I_{p}(T_0)&=\frac{\mathtt{c}}{p^{2}},
\end{aligned}
\end{equation}
with
\begin{equation}\label{timediff}
\frac{\sqrt{2}^{p+1}}{\mathtt{c}^{\tfrac{p-1}{2}}} \,p^{-1} \le T_0\le \frac{\sqrt{2}^{p+1}}{\mathtt{c}^{\tfrac{p-1}{2}} \left( 1-p^{-1} \right)^{p/2}} \,p^{-1}.
\end{equation}

\end{lemma}

\begin{proof}
We consider the following linear symplectic change of coordinates
\[
\varphi_1=\theta_1, \quad \varphi_p=-p \,\theta_1+\theta_p, \quad J_1=I_1+p\,I_p, \quad J_p=I_p.
\]
The new Hamiltonian reads as
\[
\mathcal{G}_*=\lambda\,(J_1-p J_p)^{p/2} \sqrt{J_p}\,\cos(\varphi_{p}).
\]
Since $J_1=\mathtt{M}_+=I_1+p I_p$ is a constant of motion for $\mathcal{G}_*$ we can look for solutions traveling along the diffusion channel
\begin{equation}\label{fix}
\{ (\mathtt{c}-p I_p, I_p) : I_p\in (0, \mathtt{c}/p)  \}=\{ J_1=\mathtt{M}_+=\mathtt{c} \}.
\end{equation}

We fix the section $\{\varphi_{p}=\pi/2\}$, which  is invariant by the flow of $\mathcal{G}_*$.
 The dynamics of $J_p=I_p$ is determined by the equation
\[
\dot{J}_{p}=\lambda\,\,(\mathtt{c}-p J_p)^{p/2} \sqrt{J_p}=:f(J_p).
\]
The function $f$ is Lipschitz continuous and strictly positive in the interval $[\varepsilon, \mathtt{c}/p)$.
Hence we can easily conclude that there exists an orbit with initial condition $J_{p}(0)=\varepsilon$ which is monotone increasing and reach the value $\mathtt{c}/p^{2}$
at time 
\[
T_0:= \sqrt{2}^{p-1} \int_{\textstyle{\varepsilon}}^{\tfrac{\mathtt{c}}{p^{2}}} \frac{1}{(\mathtt{c} -p\,J_{p})^{p/2}\, \sqrt{J_{p}}}\, d J_{p}.
\]
By using the fact that
\[
\mathtt{c}-p \varepsilon\geq \mathtt{c} -p\,J_{p}\geq \mathtt{c}(1-p^{-1})
\]
on the interval of integration we get the bounds \ref{timediff}.

\end{proof}

\begin{remark}\label{remarkfast}
The time of diffusion $T_0$ has lower and upper bounds that do not depend on $\varepsilon=I_p(0)$, see the first line in \ref{ic}.
\end{remark}
\begin{remark}\label{supactions}
By \ref{fix} and the fact that $J_{p}=I_{p}$ is monotone increasing in the time interval $[0, T_0]$ we have that
\[
\sup_{t\in [0, T_0]} I_1(t)= I_1(0)< \mathtt{c}.
\]
\end{remark}
By the discussion above it is clear that Lemma \ref{sistdin} provides also an orbit $g^-_{\varepsilon, \mathtt{c}}(t)$ for $\mathcal{G}^-$ with the same evolution of the actions $I_{-1}, I_{-p}$ as in \ref{ic}.
We consider the family of solutions $g^{\pm}_{\varepsilon, \mathtt{c}}(t)$ given by the Lemma \ref{sistdin} with
\begin{equation}\label{parameters}
\begin{aligned}
 {\mathtt{c}\in [\mathtt{c}_-, \mathtt{c}_+]}, \qquad
\varepsilon>0,
\end{aligned} 
\end{equation}
where $\tc_{\pm}$ and $\varepsilon$ shall be chosen later and shall depend respectively on $p$ and $(p, s)$.

We define $b(\varepsilon, \mathtt{c}; t, x)=b(t, x)=\sum_{j\in\mathbb{Z}} b_j(t) e^{\mathrm{i} j x}$ with 
\[
b_j(t):=
\begin{cases}
\sqrt{I_j(t)} \,e^{\mathrm{i}\theta_j(t)} \quad j\in S\\
0 \qquad \qquad \quad \,\,\,\,\,\text{otherwise}.
\end{cases}
\]
Then the function $b(t, x)$ is a solution of $H_{\mathrm{res}}$ such that
\begin{equation}\label{energy}
\begin{aligned}
|b_{\pm 1}(0)|^2&=\mathtt{c}- \varepsilon\,p, &  |b_{\pm p}(0)|^2&=\varepsilon,\\
 |b_{\pm 1}(T_0)|^2&=\mathtt{c} (1-p^{-1}),  & |b_{\pm p}(T_0)|^2&=\frac{\mathtt{c}}{p^{2}}.
\end{aligned}
\end{equation}
By \ref{timediff} we have
\begin{equation}\label{Tzero}
\frac{\sqrt{2}^{p+1}}{p\sqrt{\tc_+}^{p-1}}\le T_0\le  \frac{ \sqrt{2}^{p+3}}{p \sqrt{\tc_-}^{p-1}}.
\end{equation}

\subsection{Approximation argument}\label{sec:approx}
In this section we show that solutions of the Hamiltonian \ref{hamtime} whose initial conditions are $\ell^1$-close enough to the initial datum of (an opportune rescaling of) $b(t, x)$ remain $\ell^1$-close to it for all the time that we need to appreciate the drift in the actions \ref{energy}. 

\smallskip

 The solutions $u(t, x)$ of $H_{\mathrm{res}}$ are invariant under the rescaling
\[
u(t, x)\to \mu^{-1} u(\mu^{-p+1} t, x).
\]
Given $\mu>0$ we consider the rescaled solution 
\begin{equation}\label{rescaling}
r^{\mu}(t, x):=\mu^{-1} b(\mu^{-p+1} t, x).
\end{equation}
We remark the dependence of $r^{\mu}(t, x)=r^{\mu}(\varepsilon, \tc; t,x)$ on the parameters $\varepsilon$, the initial value of the energy of the high modes, and $\tc$, the height of the partial momenta.
The diffusion time is rescaled in the following way
\begin{equation}\label{time}
 T:=\mu^{p-1} T_0 ,
\end{equation}
hence we need to ensure a good approximation of $r^{\mu}(t)$ by solutions $r(t)$ of \ref{hamtime} at least in the range of time $[0, T]$.
By Remark \ref{supactions} we have
\begin{equation}\label{stimabase}
|| r^{\mu}(t) ||_{\ell^1}\le 4 \sqrt{\mathtt{c}} \mu^{-1} \qquad \text{for}\,\,t\in [0, T].
\end{equation}

Let us define (recall $\gamma$ in \ref{irrational})
\begin{equation}\label{ro0}
\begin{aligned}
&\mu_0:=\,p^{\mathtt{a}} \gamma^{-\mathtt{b}}, \quad \mathtt{a}>0,  \quad \mathtt{b}:= \frac{1}{\frac{p}{4}-2}, \quad \mathtt{r}:=\frac{\tc_+}{\tc_-},
\end{aligned}
\end{equation}
{where $\mathtt{a}$ is a constant to be chosen large enough.}
 \begin{proposition}\label{approxarg}
Let $\varepsilon>0$. There exists a universal constant $p_0>0$ such that for all $p\geq p_0$ the following holds. If we consider
\begin{align}
&\mathtt{a}\geq \, p\,p!\, 4^{p+1} , \qquad \tc\in [\tc_-, \tc_+],   \label{condparam1}\\[2mm]   \label{condparam2}
& \mathtt{r}^{\frac{p-1}{2}}=2 , \quad \tc_-\geq 1, \quad \tc_+\le 2^{2-p^{-1}} (p!)^{2 p^{-1}} p^{\frac{\mathtt{a}}{2} -14},
\end{align}
then
for all $\mu\geq \mu_0=\mu_0(\mathtt{a}, p)$ in \ref{ro0} we have that if $r(t)$ is a solution of \ref{hamtime} with initial datum satysfying 
\begin{equation}\label{sigma1}
 || r(0)-r^{\mu}(\varepsilon, \tc; 0) ||_{\ell^1}\le \mu^{-\sigma_1}, \quad \sigma_1:=\frac{3 p}{4}+2 , 
 \end{equation}
 then
\begin{equation}\label{sigma2}
|| r(t)-r^{\mu}(\varepsilon, \tc; t) ||_{\ell^1}\le 2 \mu^{-\sigma_2}, \quad \sigma_2:=\sigma_1-\frac{1}{2}, \quad \text{for}\,\,t\in [0, T].
\end{equation}
\end{proposition}

\begin{proof}
We set $\xi:=r-r^{\mu}$ and we study the evolution of its $\ell^1$-norm. We observe that $\Pi_S^{\perp}\xi=\Pi_S^{\perp} r$.
We have that $\dot{\xi}=Z_0(t)+Z_1(t)\,\xi+Z_2(t, \xi)$ with (recall \ref{hamrotating})
\begin{align*}
&Z_0:=X_{\mathcal{R}}(r_{\mu}), \\
&Z_1:= D X_{H_{\mathrm{res}}}(r^{\mu}),\\
&Z_2:=X_{H_{\mathrm{res}}}(r)-X_{H_{\mathrm{res}}}(r^{\mu})-D X_{H_{\mathrm{res}}}(r^{\mu})\xi+X_{\mathcal{R}}(r)-X_{\mathcal{R}}(r_{\mu})+X_{\mathcal{Q}}(r)-X_{\mathcal{Q}}(r_{\mu}).
\end{align*}
By the differential form of Minkowsky's inequality we get
\[
\frac{d}{d t} || {\xi} ||_{\ell^1}\le ||Z_0(t)||_{\ell^1}+||Z_1(t)\,\xi ||_{\ell^1}+||Z_2(t)||_{\ell^1}.
\]
Under suitable conditions on the parameters $\mathtt{a}, \tc_{\pm}, p$, we provide bounds on the terms of the right hand side of the above inequality. Later we shall prove that such conditions are satisfied by the assumptions \ref{condparam1}, \ref{condparam2}.
Let us denote
\begin{equation}\label{def:Xi}
\Xi=\Xi( p, \tc_-):=\,\sqrt{2}^{3p-1}\,p^2 p! \sqrt{\tc_-}^{p-1}.
\end{equation}
We shall impose
\begin{itemize}
\item  an upper bound for $\tc_+$ to obtain the bound of $Z_0$;
\item  an upper bound on $\mathtt{r}$ to obtain the bound of $Z_1$;
\item  an upper bound on $\mathtt{r}$ and a lower bound for $\tc_-$ to obtain the bound of $Z_2$.
\end{itemize}

\smallskip

\noindent\textbf{Bound for $Z_0$}: By \ref{boundR}, \ref{stimabase} we have
\[
|| Z_0||_{\ell^1}\le C_1 \gamma^{-1}\,(4 \sqrt{\mathtt{c}})^{2p-1}\, \mu^{-2p+1}+\widetilde{C_1} \gamma^{-2}\,(4 \sqrt{\mathtt{c}})^{3p-2}\, \mu^{-3p+2}.
\]
\begin{itemize}
\item We claim that, under the following conditions
\begin{align}
\mathtt{r}^{\frac{1-p}{2}}\,\sqrt{2}^{3(p-1)}\,\,p!\,p^{\mathtt{a} (\frac{p}{4}-2)-7p}  &\geq \sqrt{\tc}_+^{p}   ,  \label{ap1}\\
\mathtt{r}^{\frac{1-p}{2}}\,\sqrt{2}^{3(p-1)}\,\,p!\,p^{\mathtt{a} (\frac{5 p}{4}-3)-\frac{21\,p}{2}-\frac{3}{2} }  &\geq  \sqrt{\tc}_+^{2p-1}  ,  \label{ap2}
\end{align}
we have
\begin{equation}\label{bound:Z0}
|| Z_0 ||_{\ell^1}\le \Xi \mu^{-(7/4)p-1}.
\end{equation}
\end{itemize}
We want to impose that
\begin{equation}\label{power}
\begin{aligned}
C_1 \,(4 \sqrt{\mathtt{c}})^{2p-1}\, \gamma^{-1} \mu^{-2p+1} &\le (\Xi/2) \mu^{-(7/4) p-1}, \\
 \widetilde{C_1} \gamma^{-2} \,(4 \sqrt{\mathtt{c}})^{3p-2}\, \mu^{-3p+2} &\le (\Xi/2) \mu^{-(7/4) p-1}.
 \end{aligned}
\end{equation}
{The first inequality in \ref{power} is equivalent to  
$
\mu^{\frac{p}{4}-2} \gamma \geq (2/\Xi) C_1 \,(4 \sqrt{\mathtt{c}})^{2p-1}\,
$.\\ Thus we ask for 
\begin{equation}\label{ap01}
\begin{aligned}
\mu_0^{\frac{p}{4}-2} \gamma=p^{\mathtt{a} (p/4-2)} \gamma^{1-\mathtt{b}(\frac{p}{4}-2)}\stackrel{\ref{ro0}}{=} p^{\mathtt{a} (p/4-2)} &\geq (2/\Xi) C_1 \,(4 \sqrt{\mathtt{c}})^{2p-1}\,\\
&\stackrel{\ref{C1tildeC1}}{=} (1/\Xi)\,2^{p} p^{3} ((p+1)!)^2 \,(4 \sqrt{\mathtt{c}})^{2p-1}\,.
\end{aligned}
\end{equation}
If $p\geq 4$ then
\[
2^{p} p^{3} ((p+1)!)^2 \,(4 \sqrt{\mathtt{c}})^{2p-1}\le p^{7p+2} \,\sqrt{\mathtt{c}_+}^{2p-1}.
\]
Therefore condition \ref{ap1} implies \ref{ap01}.\\
The second inequality in \ref{power} is equivalent to  
$
\mu^{\frac{5 p}{4}-3} \gamma^2 \geq (2/\Xi)\,\widetilde{C_1} (4 \sqrt{\mathtt{c}})^{3p-2}
$. Then it is sufficient to ask for 
\begin{equation}\label{ap02}
\mu_0^{\frac{5 p}{4}-3} \gamma^2=p^{\mathtt{a} (\frac{5 p}{4}-3)} \gamma^{2-\mathtt{b}(\frac{5 p}{4}-3)} \geq (2/\Xi) \widetilde{C_1} (4 \sqrt{\mathtt{c}})^{3p-2}
\end{equation}
where
\[
(2/\Xi) \widetilde{C_1} (4 \sqrt{\mathtt{c}})^{3p-2}\stackrel{\ref{C1tildeC1}}{=} (2/\Xi)(3p-1)\,p^{5} 2^{\frac{3}{2} p-\frac{5}{2}}\, ((p+1)!)^3 (4 \sqrt{\mathtt{c}})^{3p-2}.
\]
By the definition of $\mathtt{b}$ in \ref{ro0}, we have $2-\mathtt{b}(\frac{5 p}{4}-3)<0$, thus we can disregard the presence of $\gamma$ in the above inequality. If $p\geq 4$ then 
\[
(3p-1)\,p^{5} 2^{\frac{3}{2} (p-1)}\, ((p+1)!)^3 (4 \sqrt{\mathtt{c}})^{3p-2}\le p^{\frac{7}{2} (3p+1)}  \sqrt{\mathtt{c}_+}^{3p-2}.
\]
Then condition \ref{ap2} implies \ref{ap02}.}
This proved the claim.

\smallskip

\noindent\textbf{Bound for $Z_1$}: 
Assume that 
\begin{equation}\label{ratiocs}
\mathtt{r}^{\frac{p-1}{2}}\le \frac{2\,p \,p! }{ (p-1)}  . 
\end{equation} 
Taking into account the definition of $H_{res}$ in \ref{resham}, the bound \ref{stimabase} and Lemma \ref{lemmaGT} we have
\begin{equation}\label{bound:Z1}
|| Z_1(t)\, \xi ||_{\ell^1}\le p (p-1) \sqrt{2}^{-(p+1)} (4 \sqrt{\tc})^{p-1} \mu^{ -p+1} ||\xi||_{\ell^1} \le \Xi  \mu^{ -p+1} ||\xi||_{\ell^1}
\end{equation}
provided that 
\begin{equation}
p (p-1) \sqrt{2}^{-(p+1)} (4 \sqrt{\tc})^{p-1}\le \Xi.
\end{equation}
This is implied by \ref{ratiocs}.

\smallskip

\noindent\textbf{Bound for $Z_2$}: We use a bootstrap argument. Let us define $T_*$ as the sup of the times $t$ such that 
\begin{equation}\label{bootass}
|| \xi(t)||_{\ell^1}\le 2 \mu^{-\sigma_2}.
\end{equation}
We observe that for $t=0$ we have $|| \xi(0) ||_{\ell^1}\le \mu^{-\sigma_1}$ and, since $\sigma_1-\sigma_2=1/2$ and $\mu$ will be taken large enough, we have $T_*>0$. A posteriori we shall prove that $T_*>T>0$. We call
\[
Z_{2, 1}:=X_{H_{\mathrm{res}}}(r)-X_{H_{\mathrm{res}}}(r^{\mu})-D X_{H_{\mathrm{res}}}(r^{\mu})\xi,
\]
\[
Z_{2, 2}:=X_{\mathcal{R}}(r)-X_{\mathcal{R}}(r_{\mu}), \quad Z_{2, 3}:=X_{\mathcal{Q}}(r)-X_{\mathcal{Q}}(r_{\mu}).
\]
\begin{itemize}
\item We claim that, under the following condition
\begin{equation}\label{ratiocs2}
\,\sqrt{\tc_-}\geq  \mathtt{r}^{\frac{p-1}{2}}  \dfrac{(p^2-1) 2^{p-5} }{p! }      
\end{equation}
we have
\begin{equation}\label{bound:Z21}
|| Z_{2, 1} ||_{\ell^1}\le \Xi  \mu^{ -p+1} ||\xi||_{\ell^1}.
\end{equation}
\end{itemize}

By the definition of $H_{res}$ in \ref{resham}
\begin{align*}
|| Z_{2, 1} ||_{\ell^1} \le& p^2 (p+1)\, \sqrt{2}^{-(p+1)} \sum_{j=2}^{p} || r^{\mu} ||_{\ell^1}^{p-j} ||\xi ||_{\ell^1}^j \\
\stackrel{\ref{bootass}, \ref{stimabase}}{\le}& p^2 (p+1) \, \sqrt{2}^{-(p+1)}2^{p-1}  (4\sqrt{\mathtt{c}})^{p-2}\sum_{j=2}^{p} \mu^{(j-p)-\sigma_2(j-1)} || \xi ||_{\ell^1}  .
\end{align*}
Since $\sigma_2\geq 1$ {we have $\mu^{(j-p)-\sigma_2(j-1)}\le \mu^{-p+1}$ for $j\geq 2$ and
\[
|| Z_{2, 1} ||_{\ell^1}\le p^2 (p+1) \, \sqrt{2}^{-(p+1)}2^{p-1}  (4\sqrt{\mathtt{c}})^{p-2}\,(p-1)  \,\mu^{-p+1} ||\xi ||_{\ell^1}\le \Xi  \mu^{ -p+1} ||\xi||_{\ell^1}
\]}
provided that (recall \ref{def:Xi})
\begin{equation}\label{c+c-}
p^2 (p^2-1) \, \sqrt{2}^{5p-11}  \sqrt{\mathtt{c}}^{p-2}\,\le \Xi.
\end{equation}
This is implied by the bound \ref{ratiocs2}.
Now recall the bound \ref{boundR}. We have
\begin{equation}\label{bZ22}
|| Z_{2, 2} ||_{\ell^1}\le p^2 C_1 \gamma^{-1} \sum_{j=1}^{2p-1} || r^{\mu} ||_{\ell^1}^{2p-1-j} ||\xi ||_{\ell^1}^j +\widetilde{C_1} p^2 \gamma^{-2} \sum_{j=1}^{3p-2} || r^{\mu} ||_{\ell^1}^{3p-2-j} ||\xi ||_{\ell^1}^j .
\end{equation}
{
We reason as for the bound on $Z_{2, 1}$. We shall use systematically bounds \ref{stimabase}, \ref{power}, \ref{bootass}. 
\begin{itemize}
\item We claim that, under the following conditions (recall \ref{def:Xi})
\begin{align}
\sqrt{\tc_-} \geq& 2^{2p} p^{2-\mathtt{a}(\frac{p}{4}+1)} (2p-1) , \label{bromuro}\\
\sqrt{\tc_-} \geq& 2^{3p-5} p^{2-\mathtt{a}(\frac{p}{4}+1)}(3p-2) . \label{bromuro2}
\end{align}
we have
\begin{equation}\label{bound:Z22}
|| Z_{2, 2} ||_{\ell^1}\le \Xi \mu^{-(3/2)p+1} || \xi ||_{\ell^1}.
\end{equation}
\end{itemize}
We deal with the first term in the r.h.s of \ref{bZ22}. We have $\textstyle{\mu^{-2p+1+j-\sigma_2(j-1)}\le \mu^{-2p+2}}$, because $\textstyle{\sigma_2\geq 1}$, then
\[
 \sum_{j=1}^{2p-1} || r^{\mu} ||_{\ell^1}^{2p-1-j} ||\xi ||_{\ell^1}^j\le 2^{2p-2} (4 \sqrt{\tc})^{2p-2} (2p-1) \mu^{-2p+2} || \xi ||_{\ell^1}.
\]
We want to prove that $\textstyle{C_1 \gamma^{-1} 2^{2p-2} (4 \sqrt{\tc})^{2p-2}  (2p-1) p^2 \mu^{-2p+2}\le (\Xi/2) \mu^{-(3/2)p+1}}$. By using \ref{power}, it is easy to check that this is implied by condition \ref{bromuro}.
Regarding the second term in the r.h.s of \ref{bZ22}, we have $\textstyle{\mu^{-3p+2+j-\sigma_2(j-1)}\le \mu^{-3p+3}}$, because $\textstyle{\sigma_2\geq 1}$, then
\[
\sum_{j=1}^{3p-2} || r^{\mu} ||_{\ell^1}^{3p-2-j} ||\xi ||_{\ell^1}^j\le 2^{3p-3} (4 \sqrt{\tc})^{3p-3}  (3p-2) \mu^{-3p+3} || \xi ||_{\ell^1}.
\]
We want to prove that $p^2\widetilde{C_1} \gamma^{-2} 2^{3p-3} (3p-2) (4 \sqrt{\tc})^{3p-3} \mu^{-3p+3}\le (\Xi/2) \mu^{-(3/2)p+1}$. By using \ref{power}, one can check that this is implied by \ref{bromuro2}.
This proves the claim.\\
}
The most problematic term is {$H^{(p+1, \geq 2)}$}, because it has the same degree of $H_{\mathrm{res}}$. However we recall that the monomials of {$H^{(p+1, \geq 2)}$} are Fourier supported on at least {two} normal modes. 
\begin{itemize}
\item We claim that, under the conditions
\begin{align}
\mathtt{r}^{\frac{p-1}{2}}\le & 2, \label{pallino}\\
\sqrt{\tc_-}\geq & \mathtt{r}^{\frac{p-1}{2}}\, 2^{-(3p+4)}, \label{pallino2}
\end{align}
we have
\begin{equation}\label{bound:Z23}
|| Z_{2, 3} ||_{\ell^1}\le \Xi \mu^{-p+1} || \xi ||_{\ell^1}.
\end{equation}
\end{itemize}
By the \ref{ham}, \ref{hamrotating}, the bound \ref{boundcoeff} and by noting that $ || \Pi_S^{\perp} r ||_{\ell^1}\le  || \xi ||_{\ell^1}$ we have
\begin{align*}
|| Z_{2, 3} ||_{\ell^1}\le& p^2 p! \sqrt{2}^{-(p+1)} \sum_{{j=1}}^{p} || r^{\mu} ||_{\ell^1}^{p-j} ||\xi ||_{\ell^1}^j   \\
\stackrel{\ref{bootass}, \ref{stimabase}}{\le}&      p^2 p! \sqrt{2}^{-(p+1)} 4^{p-1} \sqrt{\tc}^{p-1} \mu^{1-p} || \xi ||_{\ell^1}\\
+&p^2 p! \sqrt{2}^{-(p+1)} 4^{p-2} \sqrt{\tc}^{p-2} 2^{p-1} ||\xi ||_{\ell^1} \sum_{j=2}^p \mu^{j-p-\sigma_2(j-1)}   \\
\le& p^2 p! \sqrt{2}^{-(p+1)} 4^{p-1} \sqrt{\tc}^{p-1} \mu^{1-p} || \xi ||_{\ell^1}\\
+&p^3 p! \sqrt{2}^{-(p+1)} 4^{p-2} \sqrt{\tc}^{p-2} 2^{p-1}  \mu^{1-p} ||\xi ||_{\ell^1}.
\end{align*}
We obtain the bound \ref{bound:Z23} provided that
\begin{equation}
p^2 p! \sqrt{2}^{-(p+1)} 4^{p-1} \sqrt{\tc}^{p-1}\le \frac{\Xi}{2}, \qquad p^3 p! \sqrt{2}^{-(p+1)} 4^{p-2} \sqrt{\tc}^{p-2} 2^{p-1}\le \frac{\Xi}{2}.
\end{equation}
Those are implied by \ref{pallino}, \ref{pallino2} respectively.
By collecting the previous estimates \ref{bound:Z0}, \ref{bound:Z1}, \ref{bound:Z21}, \ref{bound:Z22}, \ref{bound:Z23} we obtained {
\[
\frac{d}{dt} || \xi ||_{\ell^1}\le \Xi \Big( \mu^{-(7/4)p-1}+\,  \mu^{ -p+1} ||\xi||_{\ell^1} \Big).
\]
}

Then
by Gronwall lemma 
\[
|| \xi(t) ||_{\ell^1} \le 2  {\mu}^{-\sigma_1} \,\exp( \Xi \mu^{-p+1} t) \quad \text{for}\,\,t\in[0, T_*].
\]

For times $t\in[0, c_0\mu^{p-1} \log(\mu)]$ with 
\begin{equation}\label{c0upper}
 c_0 := \frac{1}{2  \Xi}    
\end{equation}
 we have that $|| \xi ||_{\ell^1}\le 2 \mu^{ \Xi c_0-\sigma_1}\le 2 \mu^{-\sigma_2}$. Then $T_*\geq c_0 \mu^{p-1} \log(\mu)$.
We now prove that $c_0 \mu^{p-1} \log(\mu)>T$. Then $T_*>T$ and we can drop the bootstrap assumption.  Recalling \ref{Tzero}, \ref{parameters} we prove that 
\[
c_0 \log(\mu_0)\geq  \frac{2 \sqrt{2}^{p+1}}{p \sqrt{\tc_-}^{p-1}}\geq T_0.
\]
 By using the definition of $\mu_0$ in \ref{ro0} and \ref{condparam1} we have
\begin{equation}\label{c0lower}
c_0 \log(\mu_0)=\mathtt{a} c_0 \log(p \gamma^{-\frac{\mathtt{b}}{\mathtt{a}}})> \mathtt{a}\,c_0\stackrel{\ref{c0upper}}{=}\frac{\mathtt{a}}{2 \Xi}\stackrel{\ref{condparam1}}{\geq}  \frac{2 \sqrt{2}^{p+1}}{\,p \sqrt{\tc_-}^{p-1}}.  
\end{equation}
We conclude the proof by showing that taking $\mathtt{a}, \tc_{\pm}$ as in \ref{condparam1}, \ref{condparam2} the conditions \ref{ap1}, \ref{ap2}, \ref{ratiocs}, \ref{ratiocs2}, \ref{bromuro}, \ref{bromuro2}, \ref{pallino}, \ref{pallino2} are satisfied.\\
The conditions \ref{pallino}, \ref{pallino2} imply respectively \ref{ratiocs}, \ref{ratiocs2} if $p$ is taken large enough .\\
By the choice of $\mathtt{r}$ in \ref{condparam2}  the inequality \ref{pallino} is satisfied.\\
If $\tc_-\geq 1$ it is easy to see that, for $p$ large enough, the bounds \ref{pallino2}, \ref{bromuro} and \ref{bromuro2} hold.\\
We prove that, assuming \ref{condparam1}, \ref{condparam2} and taking $p$ large enough, imply \ref{ap1}, \ref{ap2}.\\ The condition \ref{ap1} is equivalent to
\begin{equation}\label{app1}
\tc_+\le 2^{3-p^{-1}} (p!)^{2/p} p^{\mathtt{a} (\frac{1}{2}-\frac{4}{p})-14}
\end{equation}
While \ref{ap2} is equivalent to
\begin{equation}\label{app2}
\tc_+\le 2^{3-\frac{p}{2p-1}} (p!)^{\frac{2}{2p-1}} p^{\mathtt{a} (\frac{5p-12}{4p-2})-\frac{14 p+3}{2p-1}}
\end{equation}
For $p$ large enough the right hand side of \ref{app2} is greater than the r.h.s of \ref{app1}. To obtain the upper bound of $\tc_+$ in \ref{condparam2} from \ref{app1} we consider that, if $p\geq 16$ then $p^{4/p}\le 2$ and this implies
\[
p^{\mathtt{a} (\frac{1}{2}-\frac{4}{p})}\geq \frac{p^{\mathtt{a}/2}}{2}.
\]
\end{proof}

\begin{remark}\label{rmkupperboundc}
The upper bound for $\tc_+$ is not optimal and the interval $[\tc_-, \tc_+]$ can be enlarged. However the approximation argument does not hold for $\tc_+=\mathcal{O}(p^{2\mathtt{a}})$ (see for instance the proof of the bound \ref{bound:Z0} for $Z_0$). This fact is fundamental in the estimate \ref{normT} for the norm at time $T$ of the unstable solution. See Remark \ref{rmk:weak}.
\end{remark}

\subsection{Conclusion of the proof}\label{secconclusion}
In this section we conclude the proof of Theorem \ref{thmwave} by showing that a solution $z(t)$ with initial datum $z(0)=r^{\mu}(0)$, with an opportune choice of $\mu$, undergoes the prescribed growth of its Sobolev norms.\\
We fix $\delta\ll 1$, $\mathcal{C}\gg 1$ and we consider $s>2$.
Recalling the rescaling \ref{rescaling} we consider $\mu=p^{\mathtt{a}}\gamma^{-\mathtt{b}}$ with $\mathtt{a}:=\, p\,p!\,4^p$ (see \ref{condparam1}) and $\mathtt{b}$ as in \ref{ro0}. We consider $\tc\in[\tc_-, \tc_+]$ as in \ref{condparam2}. To simplify the exposition we fix 
\begin{equation}\label{c+}
\tc_+:= p^{\frac{\mathtt{a}}{4} }.
\end{equation}

Let us consider $r(t)$ solution of \ref{hamtime} with $r(0)=\Gamma^{-1} r^{\mu}(0)$. Thanks to the choice of $\mu$ as above we can apply the approximation argument in Proposition \ref{approxarg}.
Let us call
\[
z(t)=\Gamma\big( (r_j e^{\mathrm{i} \omega(j) t})_{j\in\mathbb{Z}} \big).
\]
Now we give a lower bound for $|| z(T) ||_s$. It turns out that it is sufficient to estimate $|z_{\pm p}(T)|$. We give a lower bound for $z_p(T)$, the one for $z_{-p}(T)$ is obtained in the same way.
We have
\begin{equation}\label{triangle}
\begin{aligned}
|z_{p}(T)| &\geq |r_{p}(T)|-|\Gamma_p \big( (r_j e^{\mathrm{i} \omega(j) t})_{j\in\mathbb{Z}} \big)-r_{p}(T) e^{\mathrm{i} \omega(p) T}|\\
&\geq  |r^{\mu}_{p}(T)|-|r_{p}(T)-r^{\mu}_{p}(T)|-|\Gamma_p \big( (r_j e^{\mathrm{i} \omega(j) t})_{j\in\mathbb{Z}} \big)-r_{p}(T) e^{\mathrm{i} \omega(p) T}|.
\end{aligned}
\end{equation}
First we need a lower bound for $|r^{\mu}_{p}(T)|$. 
By \ref{energy} and the rescaling \ref{rescaling} {
\begin{equation}\label{b1}
 |r^{\mu}_{p}(T)|\geq \sqrt{\mathtt{c}}  p^{-1}\, \mu^{-1}.
\end{equation}
}
Now we give an upper bound for $|r_{p}(T)-r^{\mu}_{p}(T)|$.
By the estimates \ref{gammaid} and \ref{stimabase} we have that
\[
|| r(0)-r^{\mu}(0) ||_{\ell^1} \le C_0 \gamma^{-1} (4 \sqrt{\tc})^p \mu^{-p}.
\]
{By using that $\mu=p^{\mathtt{a}}\gamma^{-\mathtt{b}}$ it is easy to see that $|| r(0)-r^{\mu}(0) ||_{\ell^1}\le {\mu}^{-\sigma_1}$ (recall $\sigma_1$ in \ref{sigma1}), provided that 
\begin{equation}\label{cond}
p^{\mathtt{a} (\frac{p}{4}-2)-\frac{11}{4} p-\frac{7}{4}}\geq \sqrt{\tc_+}^{p} .
\end{equation}

The bound \ref{cond} is implied by the choice of $\tc_+$ in \ref{c+} if $p$ is large enough (recall $\mathtt{a}\sim p!$).
Then by Proposition \ref{approxarg} (recall $\sigma_2$ in \ref{sigma2}) 
\begin{equation}\label{diff}
||  r(t)-r^{\mu}(t) ||_{\ell^1}\le 2\mu^{-\sigma_2} \quad \text{for}\,\,t\in [0, T].
\end{equation}
}
Hence
\begin{equation}\label{b2}
|r_{p}(T)-r^{\mu}_{p}(T)|\le 2\mu^{-\sigma_2}.
\end{equation}
We are left with an upper bound for $|\Gamma_p \big( (r_j e^{\mathrm{i} \omega(j) t})_{j\in\mathbb{Z}} \big)-r_{p}(T) e^{\mathrm{i} \omega(p) T}|$.
By \ref{diff} and \ref{stimabase} we have
\[
|r_{p}(T)|\le || r(T)||_{\ell^1}\le || r^{\mu}(T)||_{\ell^1}+|| r(T)-r^{\mu}(T) ||_{\ell^1}\le 8\sqrt{\mathtt{c}} \mu^{-1},
\]
where the last inequality holds provided that
\[
\mu^{1-\sigma_2}\le 2 \sqrt{\tc_-},
\]
that is equivalent to
\[
2 p^{\frac{3}{4}\mathtt{a} (p+1)}\geq  2^{\frac{1}{p-1}} \gamma^{\mathtt{b}(\frac{3}{4}p+\frac{1}{2})}.
\]
Since $\tb>0$ this holds for $p$ large enough.
By using the estimate \ref{gammaid} 
we have
\begin{equation}\label{b3}
|\Gamma_p \big( (r_j e^{\mathrm{i} \omega(j) t})_{j\in\mathbb{Z}} \big)-r_{p}(T) e^{\mathrm{i} \omega(p) T}|\le (8\sqrt{\tc})^p\,C_0 \gamma^{-1} \mu^{-p}\le \mu^{-\sigma_1}
\end{equation}
provided that \ref{cond} holds and $p$ is large enough.
By \ref{triangle} and collecting the bounds \ref{b1}, \ref{b2}, \ref{b3}
we obtained
\[
|z_{ p}(T)|\ge \sqrt{\mathtt{c}} \,p^{-1}\,\mu^{-1}-2\mu^{-\sigma_2}-\mu^{-\sigma_1}\ge  \sqrt{\mathtt{c}} \frac{p^{-1}}{2}\mu^{-1} ,
\]
where the last inequality holds for $p\geq 4$, since $\sigma_1>3/2$.
 This implies that
\begin{equation}\label{normT}
|| z(T) ||_s^2\geq  \big(|z_{p}(T)|^2+|z_{ -p}(T)|^2 \big) \,p^{2 s}\ge (\mathtt{c}/2)\,\mu^{-2} p^{2 (s-2)}.
\end{equation}
\begin{remark}\label{rmk:weak}
Since $\mu=p^{\mathtt{a}} \gamma^{-\mathtt{b}}$, by \ref{normT} we have $|| z(T) ||_s^2 \geq  (\mathtt{c}/2)\, p^{2 (s-2-\mathtt{a})} \gamma^{2\mathtt{b}}$. If $s$ is kept fixed then we cannot ensure that taking $p$ large enough $|| z(T)||_s$ is arbitrarily large.\\ We observe that $\mathtt{\tc}\le \tc_+$ and, by Proposition \ref{approxarg}, $\tc_+$ behaves asymptotically like $\textstyle{p^{k \mathtt{a}}}$ with $k\in (0, 1)$ (see \ref{condparam2} and recall Remark \ref{rmkupperboundc}). Thus $\tc\,p^{-2\mathtt{a}}$ is not uniformly (in $p$) bounded from below. 
\end{remark}
Regarding the Sobolev norm at time zero of $z(t)$, by \ref{energy} and choosing $\varepsilon=\varepsilon(p, s)$ in \ref{parameters} such that
\begin{equation}\label{epschoice}
\varepsilon=\varepsilon_0 p^{-2s}, \qquad 0<\varepsilon_0\le \frac{\mathtt{c}_-}{1-p^{1-2s}}
\end{equation}
we have
\begin{equation}\label{norm0}
|| z(0) ||_s^2=|| r^{\mu}(0) ||^2_s=2 \mu^{-2} \big( (\mathtt{c}- \varepsilon\,p)+\varepsilon p^{2 s} \big)\le 4 \mathtt{c} \mu^{-2} .
\end{equation}
We observe that we can also provide the lower bound
\[
|| z(0) ||_s^2\geq 2\tc \mu^{-2}.
\]
Since
\[
2\sqrt{\tc} \mu^{-1}\le 2 p^{-\mathtt{a}} \sqrt{\tc_+} \gamma^{\mathtt{b}}\stackrel{\ref{c+}}{\le} 2 p^{-\frac{7}{8}\mathtt{a}},
\]
to obtain $|| z(0) ||_s\le \delta$ we impose that
\begin{equation}\label{pdelta}
F_{\gamma}(p)\le \frac{\delta}{2}\,, \quad F_{\gamma}(p):= p^{-\frac{7}{8}\,p\,p!\,4^p} \gamma^{\frac{1}{\frac{p}{2}-4}}.
\end{equation}
By \ref{normT} and \ref{norm0} the ratio between the Sobolev norms at time $t=T$ and $t=0$ has the following lower bound
\begin{equation}\label{ratioKG}
\frac{|| z(T) ||_s}{|| z(0) ||_s}\ge \frac{p^{s-2}}{2 \sqrt{2}}.
\end{equation}
To obtain \ref{factorgrowthC} we need to impose
\begin{equation}\label{pK}
p^{s-2}\geq  2 \sqrt{2} \mathcal{C}.  
\end{equation}
Recalling that $s>2$, the conditions \ref{pdelta} and \ref{pK} can be satisfied by taking $p\geq p_0$ for some $p_0=p_0(s, \delta, \mathcal{C})$ large enough.\\
By \ref{time}, \ref{Tzero}, \ref{condparam2}  we have that
\begin{equation}\label{Tfinal}
\begin{aligned}
\frac{2^{p-1}}{p} || z(0)||_s^{1-p}\le T=\mu^{p-1} T_0 & \le (\sqrt{2\mathtt{c}})^{p-1}T_0\, || z(0)||_s^{1-p} \\
& \le \frac{{2}^{p+2} }{p}  || z(0)||_s^{1-p}.
\end{aligned}
\end{equation}
This proves \ref{timeteo}.

\subsection{Proof of Theorem \ref{stronginstability}}
The proof follows the same steps of the proof of Theorem \ref{thmwave}. The difference relies in (i) the choice of the index $s$ to obtain the lower bound \ref{strongb} for the norm of the solution $z(t)$ at time $T$; (ii) the choice of $p$ in order to get a lower bound on the norm of the initial datum.\\
Let us fix $\delta\ll 1$, $\mathcal{C}\gg 1$. We choose $\mathtt{a}, \mathtt{b}$ and $\mathtt{c}$ as in previous section. We consider $\mu:=p^{\mathtt{a}}\gamma^{-\mathtt{b}}$ with $p=p(\delta, \gamma)$ large enough such that \ref{pdelta} holds and (recall the definition of $F_{\gamma}$ in \ref{pdelta})
\begin{equation}\label{lowb}
F_{\gamma}(p) \geq \frac{\delta}{2\sqrt{2}}.
\end{equation}
\begin{remark}
The conditions \ref{lowb} and \ref{pdelta} can be satisfied at the same time for $C_1\le p\le C_2$ with $C_i=C_i(\gamma, \delta)$, $i=1, 2$. We point out that $p$ can be considered larger than $p_0$ given in Proposition \ref{approxarg} if $\delta$ is taken small enough.
\end{remark}
The \ref{lowb} provides the lower bound for the norm of the initial datum 
\[
|| z(0)||_s\geq \frac{\delta}{2}.
\] We choose $s$ such that \ref{pK} holds. This holds for $s\geq s_0$ with $s_0=s_0(p, \mathcal{C})$ large enough. We observe that, from \ref{norm0}, we can ensure that $|| z(0)||_s\le \delta$ for all $s$. Indeed it is sufficient to choose $\varepsilon=\varepsilon(s, p)$ small enough. The choice of $\varepsilon$ does not affect the other parameters of the problem. This is due to the fact that the transfer of energy occurs via diffusion channels.\\
 Thanks to the lower bound on $||z(0)||_s$ we obtain
\[
|| z(T)||_s\geq \mathcal{C} || z(0)||_s \geq \frac{\mathcal{C} \delta}{2}.
\]
It is enough to choose $\mathcal{C}=2\,K\, \delta$ to conclude.\\
Now we show that there exists an unstable solution of \ref{wave} with an upper bound on the diffusion time like \ref{timeC}.
Assume that $K=\delta^{-\alpha}$ for some $\alpha>1$. If we choose 
\[
\mu^{-1}=\frac{\delta}{4 \sqrt{2\tc}}.
\]
then
\[
T=\mu^{p-1} T_0\le (\delta^{-1})^{p-1} \,8^p\,\mathtt{r}^{\frac{p-1}{2}}\,p^{-1}\le \mathcal{C}^{p-1} \delta^{\alpha(p-1)} \,\sqrt{2}^{7p-1}\,p^{-1}.
\]
We have
\[
\lim_{p\to \infty} \left(\sqrt{2}^{7p-1}\,p^{-1}\right)^{\frac{1}{p-1}}=\frac{1}{8\sqrt{2}}.
\]
Then there exists $\alpha_0>0$ a pure constant large enough such that for $\alpha\geq \alpha_0$ and $\delta$ small enough $T\lesssim \mathcal{C}^{p-1}$.

\section{Proof of Theorem \ref{thmnls}}
We follow the same steps of the proof of Theorem \ref{thmwave} shown in Section \ref{sec:wave}.
First we build the convolution potential $V_N$. To simplify the notation we write $V_N=V$.
We consider the tangential set

\begin{equation}\label{Snls}
S:=\{ k_1, k_2, k_3, k_4  \}\subset\mathbb{Z}
\end{equation} with
\begin{equation}\label{Snls2} 
k_1, k_3>0, \quad k_2<0, \quad k_4:=k_1-k_2+k_3+N.
\end{equation}
Let us also assume that
\begin{equation}\label{assummax}
k_3:=\max\{ k_1, | k_2 |, k_3\}\le \sqrt{N}.
\end{equation}
Let us consider $\mathtt{q}:=(\mathtt{q}_1, \mathtt{q}_2, \mathtt{q}_3)\in [1, 2]^3$ such that
\begin{equation}\label{irraz}
|\mathtt{q}\cdot \ell+k|\geq \frac{\gamma}{\langle \ell \rangle^{\tau}} \quad \forall \ell\in\mathbb{Z}^3, \quad 0<|\ell|\le 9, \quad \forall k\in\mathbb{Z}, \quad (\ell, k)\neq (0, 0)
\end{equation}
with $\gamma\in (0, 1)$ and $\tau>0$. It is well known that for $\tau$ large enough the set of vectors in $[1, 2]^3$ satisfying \ref{irraz} has positive measure.  We set
\[
V_j:=\begin{cases}
\mathtt{q}_i-k_i^2, \qquad\qquad\qquad j=k_i, \quad i=1, 2, 3,\\
\mathtt{q}_1-\mathtt{q}_2+\mathtt{q}_3-k_4^2, \quad \,\, \,j=k_4,\\
0 \qquad \qquad \qquad \qquad \,\, \,\,\,\text{otherwise}.
\end{cases}
\]

We consider the Fourier expansion $u=\sum_{j\in\mathbb{Z}} u_j\,e^{\mathrm{i} j x}$ with $u_j:=\frac{1}{2\pi} \int_{\mathbb{T}} u\,e^{-\mathrm{i} j x}\,dx$.
The equation \ref{nls} is Hamiltonian with respect to the symplectic structure 
\[
-\mathrm{i} \sum_{j\in\mathbb{Z}} d u_j\wedge d\overline{u_j}.
\] The Hamiltonian is given by 
$
H=H^{(2)}+H^{(4)}
$ where
\begin{equation}\label{nlsham}
\begin{aligned}
H^{(2)}(u_j, \overline{u_j}):=&\sum_{j\in\mathbb{Z}} \omega(j) u_j\,\overline{u_j} , \\
 H^{(4)}(u_j, \overline{u_j}):=&\sum_{j_1-j_2+j_3-j_4=\pm N} u_{j_1}\,\overline{u_{j_2}}\,u_{j_3}\,\overline{u_{j_4}}
\end{aligned}
\end{equation}
with
\[
\omega(j):=\begin{cases}
\mathtt{q}_j \qquad\qquad \qquad \quad j=k_i, \quad i=1, 2, 3,\\
\mathtt{q}_1-\mathtt{q}_2+\mathtt{q}_3, \qquad \,j=k_4,\\
j^2 \qquad \qquad \qquad \,\,\,\,\, \text{otherwise}.
\end{cases}
\]
\begin{remark}
As in the case of the wave equation (see Remark \ref{rmkdecouple}), the tangential frequencies are irrational real numbers while the normal ones are integers.
\end{remark}
The equation \ref{nls} can be written as an infinite dimensional system of ODEs for the Fourier coefficients
\[
-\mathrm{i} \dot{u}_j=\omega(j) u_j+\sum_{j_1-j_2+j_3-j=\pm N} u_{j_1}\,\overline{u_{j_2}}\,u_{j_3}, \qquad j\in\mathbb{Z}.
\]
We consider the solution of the linear problem
\[
-\mathrm{i} \dot{u}_j=\omega(j) u_j, \quad j\in\mathbb{Z},
\]
obtained by exciting the modes in $S$, namely
\begin{equation}\label{bif2}
w(t, x)=\sum_{k\in S} a_k \,e^{\mathrm{i} (\omega(k) t+ k\, x)}.
\end{equation}
By the definition of the linear frequencies of oscillation and by \ref{irraz} the orbit $w(\cdot , x)$ is conjugated to a quasi-periodic motion on an embedded $4$-d resonant torus that fills densely a lower dimensional manifold. 

\smallskip

As in Section \ref{sec:bnf} we construct a change of coordinates that puts (partially) in normal form the Hamiltonian \ref{nlsham}. Again we work with the $\ell^1$-topology. Recall \ref{elle1} and Definition \ref{def}.
\begin{proposition}\label{wbnfnls}
Recall \ref{nlsham}.
There exists $\eta>0$ small enough such that there exists a symplectic change of coordinates $\Gamma\colon B_{\eta}\to B_{2\eta}$ which takes the Hamiltonian $H$ into its (partial) Birkhoff normal form up to order $4$, namely
\begin{equation}\label{hamBNFcoord2}
H\circ \Gamma=H^{(2)}+H_{\mathrm{res}}+H^{(4, \geq 2)}+R
\end{equation}
where:\\
(i) the resonant Hamiltonian is given by
\begin{equation}\label{resham2}
H_{\mathrm{res}}:=\Pi_{\mathrm{Ker}} H^{(4, 0)}=2 \Re\big(u_{k_1}\,\overline{u_{k_2}}\,u_{k_3}\,\overline{u_{k_4}}\big).
\end{equation}
(ii) The remainder $R$ is such that
\begin{equation}\label{boundR2}
|| X_R ||_{\eta}\lesssim \gamma^{-1} \eta^5+\gamma^{-2} \eta^7.
\end{equation}
Moreover the map $\Gamma$ is invertible and close to the identity
\begin{equation}\label{gammaid2}
|| \Gamma^{\pm 1}-\mathrm{Id} ||_{\eta}\lesssim \gamma^{-1} \eta^{3}.
\end{equation}

\end{proposition}
\begin{proof}
The proof follows the same lines of the proof of Proposition \ref{wbnf}. Actually the proof is easier since the parameter $N$, that we need to control, does not enter in the estimates for the map $\Gamma$ and the remainder $R$. The only thing that we need to prove is that $\Pi_{\mathrm{Ker}}H^{(4, 1)}=0$ and formula \ref{resham2}.\\
If $(\alpha, \beta)\in\mathcal{A}_{4, 1}$ then 
\[
\Omega(\alpha, \beta)=\mathtt{q}\cdot \ell\pm j^2
\]
for some $\ell=\ell(\alpha, \beta)\in \mathbb{Z}^3$ with $3\le |\ell|\le 9$ and $j\notin S$. Since $j^2$ is an integer, by \ref{irraz} we have $\Omega(\alpha, \beta)\geq \gamma\,9^{-\tau}>0$. This proves that $\Pi_{\mathrm{Ker}}H^{(4, 1)}=0$. If $(\alpha, \beta)\in\mathcal{A}_{4, 0}$ by \ref{irraz} and the choice of the potential the only resonant monomials are 
$
u_{k_1}\,\overline{u_{k_2}}\,u_{k_3}\,\overline{u_{k_4}}
$ and its complex conjugate.

\end{proof}

We introduce the rotating coordinates
\[
u_j=r_j \,e^{\mathrm{i} \omega(j) t}
\]
in order to remove the quadratic part of the Hamiltonian $H\circ \Gamma$. Then $r$ satisfies the equation associated to the Hamiltonian 
\begin{equation}\label{hamtime2}
\mathcal{H}=H_{\mathrm{res}}+\mathcal{Q}(t)+\mathcal{R}(t)
\end{equation}
where
\begin{equation*}\label{hamrotatingnls}
\begin{aligned}
&\mathcal{Q}\big((r_j)_{j\in\mathbb{Z}}, t\big):=H^{(4, \geq 2)}\big( (r_j \,e^{\mathrm{i} \omega(j) t})_{j\in\mathbb{Z}}\big),\\
&\mathcal{R}\big((r_j)_{j\in\mathbb{Z}}, t\big):=R\big( (r_j \,e^{\mathrm{i} \omega(j) t})_{j\in\mathbb{Z}}\big).
\end{aligned}
\end{equation*}
We study the dynamics of the resonant Hamiltonian $H_{\mathrm{res}}$. We observe that the finite dimensional subspace
\[
\mathcal{U}_S:=\{ r\colon \mathbb{Z}\to\mathbb{C} \,|\, r_j=0\,\,\,j\notin S\}
\]
is invariant by the flow of $H_{\mathrm{res}}$. We introduce the following action-angle variables on $\mathcal{U}_S$
\[
r_{k_j}=\sqrt{I_j}\,e^{\mathrm{i} \theta_j}, \quad j=1, 2, 3, 4.
\]
The Hamiltonian $H_{\mathrm{res}}$ now reads as
\[
\mathcal{G}:=2 \sqrt{I_1\,I_2\, I_3\,I_4}\cos(\theta_1-\theta_2+\theta_3-\theta_4).
\]
\begin{remark}
The Hamiltonian $\mathcal{G}$ commutes with the mass $\mathtt{M}:=I_1+I_2+I_3+I_4$.
\end{remark}

\begin{lemma}\label{sistdinnls}
Let $\varepsilon>0$ be arbitrarily small and let $\mathtt{c}>\tfrac{8}{3} \varepsilon$. There exists an orbit of $\mathcal{G}$
\[
g_{\varepsilon, \mathtt{c}}(t)=(\theta_1(t), \dots, \theta_4(t), I_1(t), \dots, I_4(t))
\]
such that
\[
I_1(0)=I_2(0)= I_3(0)=\frac{\mathtt{c}-\varepsilon}{3}, \quad I_4(0)=\varepsilon
\]
\[
I_1(T_0)=I_3(T_0)=\frac{\mathtt{c}}{6}+\frac{2\varepsilon}{3}, \quad I_2(T_0)=\frac{\mathtt{c}}{2}-\frac{4 \varepsilon}{3}, \quad I_4(T_0)=\frac{\mathtt{c}}{6}
\]
with
\[
T_0\le \frac{6}{\mathtt{c} }.
\]
\end{lemma}
\begin{proof}
We apply the following linear symplectic change of variables
\[
\varphi=A \theta,\quad J=A^{-T}\,I, \quad A:=\begin{pmatrix}
1 & 0 & 0 & 0\\
0 & 1 & 0 & 0\\
0 & 0 & 1 & 0\\
-1 & 1 &-1 & 1
\end{pmatrix}.
\]
We observe that the matrix $A\in SL(4, \mathbb{Z})$, hence this defines a linear automorphism of the torus $\mathbb{T}$.
The new Hamiltonian is given by
\[
\mathcal{G}_*=2 \sqrt{(J_1-J_4)\,(J_2+J_4)\,(J_3-J_4)\,J_4}\cos(\varphi_4).
\]
We observe that $J_1, J_2, J_3$ are constants of motion.
Then we fix
\[
J_i=\alpha:=\frac{\mathtt{c}-\varepsilon}{3}+\varepsilon \quad i=1, 3, \qquad J_2=\beta:=\frac{\mathtt{c}-\varepsilon}{3}-\varepsilon
\]
and we look for solutions traveling along the following diffusion channel
\[
\{ ( \alpha-I_4, \beta+I_4, \alpha-I_4, I_4  ) : I_4\in (0, \alpha)  \},
\]
which is contained in the mass level
\begin{equation}\label{massls}
\{ J_1+J_2+J_3=\mathtt{M}=\mathtt{c} \}.
\end{equation}
If we restrict to the invariant section $\{ \varphi_4=\pi/2 \}$ the equation of motion for $J_4=I_4$ is
\[
\dot{J_4}=2 \left(\alpha-J_4 \right) \sqrt{\,\left(\beta+J_4\right)\,\,J_4}.
\]
Reasoning as in Lemma \ref{sistdin} we can conclude that there exists an orbit such that
\[
J_4(0)=\varepsilon, \quad J_4(T_0)=\frac{\mathtt{c}}{6}
\]
with
\[
T_0=\frac{1}{2}\int^{\mathtt{c}/6}_{\varepsilon} \frac{1}{\,(\alpha-J_4) \sqrt{(\beta+J_4)\,J_4}}\,dJ_4\le \frac{3}{\mathtt{c}} \int_0^{\mathtt{c}/6} \frac{1}{\sqrt{(\frac{\mathtt{c}}{2}+J_4) \,J_4}} \,d J_4\le \frac{6}{\mathtt{c}}.
\]

\end{proof}
\begin{remark}\label{rmkepsindip}
The diffusion time $T_0$ has an upper bound that do not depend on $\varepsilon=I_4(0)$.
\end{remark}
We set 
\begin{equation}\label{parameters2}
0<\varepsilon\le \frac{\tc}{2^{2s} N^s-1}.
\end{equation}
We define $b(\varepsilon, \mathtt{c}; t, x)=b(t, x)=\sum_{j\in\mathbb{Z}} b_j(t) e^{\mathrm{i} j x}$ with 
\[
b_j(t):=
\begin{cases}
\sqrt{I_j(t)} \,e^{\mathrm{i}\theta_j(t)} \quad j\in S\\
0 \qquad \qquad \quad\,\,\,\,\, \text{otherwise}.
\end{cases}
\]
Then the function $b(t, x)$ is a solution of $H_{\mathrm{res}}$ such that
\begin{equation}\label{taorm}
\begin{aligned}
|b_{k_1}(0)|^2=|b_{k_2}(0)|^2=|b_{k_3}(0)|^2=\frac{\mathtt{c}- \varepsilon}{3}, \quad  &|b_{k_4}(0)|^2=\varepsilon,\\
 |b_{k_1}(T_0)|^2=|b_{k_3}(T_0)|^2=\frac{\mathtt{c}}{6}+\frac{2 \varepsilon}{3}, \quad |b_{k_2}(T_0)|^2=\frac{\mathtt{c}}{2}-\frac{4 \varepsilon}{3}, \quad  &|b_{k_4}(T_0)|^2=\frac{\mathtt{c}}{6}.
\end{aligned}
\end{equation}
In particular from the proof of Lemma \ref{sistdinnls} we deduce that
\begin{equation}\label{energy2}
\sup_{t\in [0, T_0]} |b_{k_i}(t)|^2<{\mathtt{c}}, \quad i=1, 2, 3, 4.
\end{equation}

We follow the same procedure of Section \ref{sec:approx}.
 The solutions $u(t, x)$ of $H_{\mathrm{res}}$ are invariant under the rescaling
\[
u(t, x)\to \mu^{-1} u(\mu^{-2} t, x).
\]
Then we consider the rescaled solution 
\begin{equation}\label{rescaling2}
r^{\mu}(t, x):=\mu^{-1} b(\mu^{-2} t, x).
\end{equation}
The diffusion time is rescaled in the following way
\begin{equation}\label{time2}
T=\mu^{2} T_0 \le \frac{6\,\mu^{2}}{\mathtt{c}}\,.
\end{equation}

By \ref{energy2} we have
\begin{equation}\label{stimabase2}
|| r^{\mu} ||_{\ell^1}\le 4 \sqrt{\mathtt{c}} \mu^{-1}. 
\end{equation}

\begin{proposition}\label{approxarg2}
Let $\gamma\in (0, 1)$ be the constant in \ref{irraz} and $\tc\geq 1$.
There exists $C_0>0$ large enough such that
for all $\mu\geq \mu_0:=C_0 \gamma^{-2} \tc^5$  
 we have the following.
If $r(t)$ is a solution of \ref{hamtime} such that 
\[
 ||  r(0)-r^{\mu}(0) ||_{\ell^1}\le \mu^{-5/2} ,
 \] then
\[
|| r(t)-r^{\mu}(t) ||_{\ell^1}\le 2 \mu^{-3/2}, \quad \text{for}\,\,t\in[0, T].
\]
\end{proposition}

\begin{proof}

We set $\xi:=r-r^{\mu}$ and we study the evolution of its $\ell^1$-norm. We observe that $\Pi_S^{\perp}\xi=\Pi_S^{\perp} r$.\\
We have that $\dot{\xi}=Z_0(t)+Z_1(t)\,\xi+Z_2(t, \xi)$ with (recall \ref{hamrotatingnls})
\begin{align*}
&Z_0:=X_{\mathcal{R}}(r_{\mu}), \\
&Z_1:= D X_{H_{\mathrm{res}}}(r^{\mu}),\\
&Z_2:=X_{H_{\mathrm{res}}}(r)-X_{H_{\mathrm{res}}}(r^{\mu})-D X_{H_{\mathrm{res}}}(r^{\mu})\xi+X_{\mathcal{R}}(r)-X_{\mathcal{R}}(r_{\mu})+X_{\mathcal{Q}}(r)-X_{\mathcal{Q}}(r_{\mu}).
\end{align*}
By the differential form of Minkowsky's inequality we get
\[
\frac{d}{d t} || {\xi} ||_{\ell^1}\le ||Z_0(t)||_{\ell^1}+||Z_1(t)\,\xi ||_{\ell^1}+||Z_2(t)||_{\ell^1}.
\]
By \ref{boundR2} and \ref{stimabase2} we have
\[
 ||Z_0(t)||_{\ell^1}\lesssim \gamma^{-1} \sqrt{\tc}^5 \mu^{-5}+\gamma^{-2} \sqrt{\tc}^7 \mu^{-7}.
\]
We impose that
\[
\gamma^{-1} \sqrt{\tc}^5 \mu^{-5}\le \mu^{-9/2}, \quad \gamma^{-2} \sqrt{\tc}^7 \mu^{-7}\le \mu^{-9/2}.
\]
Since $\tc\geq 1$ and $\gamma\in (0, 1)$, the above inequalities are satisfied if
\begin{equation}\label{cooltra}
\mu\geq \gamma^{-2} \tc^5.
\end{equation}
We obtained 
\begin{equation}\label{Lozero}
 ||Z_0(t)||_{\ell^1}\lesssim  \mu^{-9/2}.
\end{equation}
By \ref{resham2} and \ref{stimabase2} we have
\begin{equation}\label{Louno}
||Z_1(t)\,\xi ||_{\ell^1}\lesssim \tc \mu^{-2} || \xi ||_{\ell^1}.
\end{equation}
To obtain a bound for $Z_2$ we use a bootstrap argument. Let us define $T_*$ as the sup of the times $t$ such that 
\begin{equation}\label{bootass2}
|| \xi(t)||_{\ell^1}\le 2 \mu^{-3/2}.
\end{equation}
We observe that for $t=0$ we have $|| \xi(0) ||_{\ell^1}\le \mu^{-5/2}$, thus $T_*>0$. A posteriori we shall prove that $T_*>T>0$. We call
\[
Z_{2, 1}:=X_{H_{\mathrm{res}}}(r)-X_{H_{\mathrm{res}}}(r^{\mu})-D X_{H_{\mathrm{res}}}(r^{\mu})\xi,
\]
\[
Z_{2, 2}:=X_{\mathcal{R}}(r)-X_{\mathcal{R}}(r_{\mu}), \quad Z_{2, 3}:=X_{\mathcal{Q}}(r)-X_{\mathcal{Q}}(r_{\mu}).
\]
We have by \ref{bootass2}
\begin{equation}\label{Lo21}
|| Z_{2, 1} ||_{\ell^1}\lesssim \sqrt{\tc} \mu^{-1} || \xi||_{\ell^1}^2\lesssim \sqrt{\tc} \mu^{-5/2} ||\xi||_{\ell^1}\lesssim \tc \mu^{-2} ||\xi||_{\ell^1},
\end{equation}
where the last inequality holds provided that 
\begin{equation}\label{yego}
\mu\geq \tc^{-1}.
\end{equation}
This inequality is satisfied if \ref{cooltra} holds, since we assumed that $\tc\geq 1$.
In a similar way we get (see \ref{boundR2})
\[
|| Z_{2, 2} ||_{\ell^1}\lesssim \sum_{j=1}^{5} || r^{\mu} ||^{5-j} || \xi||^j_{\ell^1}\lesssim \tc^2 || \xi ||_{\ell^1} (\mu^{-4}+ \sum_{j=2}^{5} \mu^{j-5} \mu^{-\frac{3}{2}(j-1)}).
\]
Since $(1-\frac{3}{2}) j-5+\frac{3}{2}\le -4$, because $j\geq 1$, then
\begin{equation}\label{Lo22}
|| Z_{2, 2} ||_{\ell^1}\lesssim \tc^2  \mu^{-4}\,|| \xi ||_{\ell^1}\lesssim \tc \mu^{-2} || \xi ||_{\ell^1},
\end{equation}
where the last inequality holds provided that
\[
\mu\geq \sqrt{\tc}.
\]
This inequality is satisfied if \ref{cooltra} holds.
Regarding the bound for $Z_{2, 3}$, we recall that the vector field $X_{H^{(4, \geq 2)}}(r^{\mu})=0$ because $H^{(4, \geq 2)}$ is supported on at least two normal sites. We have
\begin{equation}\label{Lo23}
\begin{aligned}
|| Z_{2, 3} ||_{\ell^1}=|| X_{H^{(4, \geq 2)}} (r)  ||_{\ell^1} &\lesssim \tc \mu^{-2} || \xi ||_{\ell^1}+\sqrt{c} \mu^{-1} || \xi ||_{\ell^1}^2+|| \xi ||_{\ell^1}^3\\
&\lesssim \big(\tc \mu^{-2} +\sqrt{\tc}\mu^{-5/2}+\mu^{-9/2} \big) || \xi ||_{\ell^1}\\
&\lesssim \tc \mu^{-2} || \xi ||_{\ell^1},
\end{aligned}
\end{equation}
where the last inequality holds provided that \ref{cooltra} holds.\\
By collecting the previous estimates \ref{Lozero}, \ref{Louno}, \ref{Lo21}, \ref{Lo22}, \ref{Lo23} we obtained
\[
\frac{d}{dt} || \xi ||_{\ell^1}\le C\,\mathtt{c} \Big(\mu^{-9/2}+\mu^{-2} || \xi ||_{\ell^1}   \Big)
\]
for some pure constant $C>0$. By Gronwall Lemma we have
\[
|| \xi(t) ||_{\ell^1}\le 2\mu^{-5/2} \,e^{C\,\tc\,\mu^{-2} t}  \quad \text{for}\,\,t\in [0, T_*].
\]
For times $t\in[0, c_0\mu^{2} \log(\mu)]$ with 
\begin{equation}\label{c0upper2}
 c_0 := \frac{1}{4 C\,\tc}    
\end{equation}
 we have that $|| \xi ||_{\ell^1}\le 2 \mu^{-5/2 } \mu^{1/4}\le2 \mu^{-1/4} \mu^{-2}\le \mu^{-2}$, for $\mu$ large enough. Then $T_*>c_0 \mu^{2} \log(\mu)$.
We now prove that $c_0 \mu^{2} \log(\mu)>T$. Then $T_*>T$ and we can drop the bootstrap assumption. We have
\[
c_0  \log(\mu)= \frac{1}{4 C\,\tc}  \log(\mu)\geq \frac{6}{\tc} 
\]
if $\mu$ is large enough. Then by \ref{time2} we conclude.
\end{proof}

We fix $\delta\ll 1$, $K\gg 1$, $s> 0$, $\tc\geq 1$ and we consider $\mu=N^{\frac{3}{4} s}$. Recalling \ref{assummax} we take $N=N(\tc, \gamma, s, \delta, K)$ large enough such that
\begin{align}
N^{\frac{3}{4}s}\geq&\,\, C_0 \gamma^{-2} \tc^5,  \label{accmob}\\
\frac{1}{\sqrt{2\tc}}\, N^{\frac{s}{4}}\geq&\,\,  \delta^{-1}, \label{boh}\\
\sqrt{\frac{\tc}{6}}\,N^{\frac{s}{4}}\geq&\,\, K,  \label{boh2}
\end{align}
where $C_0$ is the constant introduced in Proposition \ref{approxarg2}.
 Let us consider $r(t)$ solution of \ref{hamtime2} with $r(0)=\Gamma^{-1} r^{\mu}(0)$. By \ref{gammaid2}, \ref{stimabase2} we have
 \[
 || r(0)-r^{\mu}(0) ||_{\ell^1}\le 64 \gamma^{-1} \tc^{3/2} \mu^{-3}.
 \]
 Then by \ref{accmob} and the definition of $\mu$, $|| r(0)-r^{\mu}(0) ||_{\ell^1}\le \mu^{-5/2}$ if $C_0$ is large enough.
 Therefore we are in position to apply Proposition \ref{approxarg2}.
 Let us call 
\[
z(t)=\Gamma\big( (r_j e^{\mathrm{i} \omega(j) t})_{j\in\mathbb{Z}} \big).
\]
Reasoning as in Section \ref{secconclusion}, we can obtain the following lower bound for the Sobolev norms of $z$ at time $t=T$ by using \ref{Snls2}, \ref{taorm}
\[
|| z(T) ||_s^2\geq  \big(|z_{k_4}(T)|^2 \big) \,k_4^{2 s}\geq \mu^{-2} \frac{\mathtt{c}}{6}\,N^{2s}= \frac{\mathtt{c}}{6}\,N^{\frac{s}{2}}\stackrel{\ref{boh2}}{\geq} K^2.
\]
Regarding the Sobolev norm at time zero, we have by \ref{taorm}  
\begin{align*}
|| z(0) ||_s^2=|| r^{\mu}(0) ||^2_s &=\big( \frac{\tc-\varepsilon}{3} (\langle k_1 \rangle^{2s}+\langle k_2 \rangle^{2s}+\langle k_3 \rangle^{2s})+\varepsilon \langle k_4 \rangle^{2s} \big)  \mu^{-2}\\
&\le \big( (\tc-\varepsilon) \langle k_3 \rangle^{2s}+\varepsilon \langle k_4 \rangle^{2s} \big)  \mu^{-2}\\
&\stackrel{\ref{assummax}}{\le} \big( (\tc-\varepsilon) N^s+\varepsilon (2N)^{2s} \big)  \mu^{-2}\\
&\stackrel{\ref{parameters2}}{\le}2\tc\, N^s\,\mu^{-2}  \stackrel{\ref{boh}}{\le}\delta^2.
 \end{align*}
We conclude by giving the estimate \ref{timecondiniziale} on the diffusion time. By \ref{time2} and \ref{boh} we get
\[
T\le \mu^2 T_0\le 12 N^{ s} || z(0)||_s^{-2}.
\]
Now we prove the bound \ref{timeC2} on the diffusion time respect to the growth.
Let us fix $s>0$, $\delta\ll 1$, $\tc=1$ and assume that $K=\delta^{-\alpha}$ with $\alpha>1$. Then $\mathcal{C}:=K/\delta=\delta^{-(1+\alpha)}$. If
\[
\delta^{-1}\geq C_0^{\frac{1}{3\alpha}}  6^{-\frac{1}{2\alpha}}  \gamma^{-\frac{2}{3\alpha}} 
\]
then the condition \ref{boh2} implies \ref{accmob}, \ref{boh}.  We can choose $N^{s/4}=\sqrt{6}\,\delta^{-\alpha}$. Therefore by \ref{time2}
\[
T\le 6\,\mu^2 T_0 = \mathcal{O}( \mathcal{C}^{\frac{6\alpha}{1+\alpha}})\lesssim \mathcal{C}^6.
\]

%

\email{Email: filippo.giuliani@upc.edu}

\end{document}